\documentclass{lmcs} 
\pdfoutput=1

\usepackage{lastpage}
\lmcsdoi{16}{1}{18}
\lmcsheading{}{\pageref{LastPage}}{}{}%
{Apr.~24,~2018}{Feb.~17,~2020}{}


\usepackage{amsmath}
\usepackage[bbgreekl]{mathbbol}
\usepackage{amssymb}
\usepackage{amsthm}
\usepackage[all]{xy}

\theoremstyle{plain}\newtheorem{theorem}[thm]{Theorem}
\theoremstyle{plain}\newtheorem{proposition}[thm]{Proposition}
\theoremstyle{plain}\newtheorem{corollary}[thm]{Corollary}
\theoremstyle{plain}\newtheorem{lemma}[thm]{Lemma}
\theoremstyle{definition}\newtheorem{definition}[thm]{Definition}
\theoremstyle{plain}\newtheorem{remark}[thm]{Remark}

\newcommand{\Nat}{{\mathbb N}}
\newcommand{\Real}{{\mathbb R}}

\newcommand{\Iii}{{\mathbb I}}

\newcommand{\Fii}{{\mathbb F}}

\newcommand{\id}{\mathrm{id}}

\newcommand{\Lip}{\mathrm{Lip}}

\newcommand{\LXY}{\Lip(X, Y)}
\newcommand{\LX}{\Lip(X)}

\newcommand{\CN}{\mathrm{C_{u}(\Nat)}}
\newcommand{\LN}{\Lip(\Nat)}

\newcommand{\lub}{\mathrm{lub}}
\newcommand{\glb}{\mathrm{glb}}

\newcommand{\Hoel}{\mathrm{{H{\oo}l}}}
\newcommand{\oo}{\mathrm{\text{\"{o}}}}

\newcommand{\Const}{\mathrm{Const}}
\newcommand{\pLip}{\mathrm{p{\text -}Lip}}

\newcommand{\RCA}{\mathrm{RCA}}

\newcommand{\TOT}{\Leftrightarrow}

\newcommand{\To}{\Rightarrow}

\newcommand{\CST}{\mathrm{CST}}

\newcommand{\CZF}{\mathrm{CZF}}

\newcommand{\BISH}{\mathrm{BISH}}

\newcommand{\PEM}{\mathrm{PEM}}

\newcommand{\WPEM}{\mathrm{WPEM}}

\theoremstyle{plain} 

\begin{document}

\title[McShane-Whitney extensions in constructive analysis]{McShane-Whitney extensions in constructive analysis}
\titlecomment{{\lsuper*}This paper is a major extension of~\cite{Pe17}.}

\author[I.~Petrakis]{Iosif Petrakis}	
\address{Mathematics Institute, Ludwig-Maximilians-Universit\"{a}t M\"{u}nchen}	
\email{petrakis@math.lmu.de}  

%





\begin{abstract}
\noindent Within Bishop-style constructive mathematics we study the classical McShane-Whitney theorem
on the extendability of real-valued Lipschitz functions defined on a subset of a metric space. Using a formulation similar to 
the formulation of McShane-Whitney theorem, we show 
that the Lipschitz real-valued functions 
on a totally bounded space are uniformly dense in the set of uniformly continuous functions. Through the
introduced notion of a McShane-Whitney pair we describe the constructive content of the original McShane-Whitney
extension and we examine how the properties of a Lipschitz function defined on the subspace of the pair extend to
its McShane-Whitney extensions on the space of the pair. Similar McShane-Whitney pairs and extensions are established for
H\"{older} functions and $\nu$-continuous functions, where $\nu$ is a modulus of continuity.
A Lipschitz version of a fundamental
corollary of the Hahn-Banach theorem, and the approximate McShane-Whitney theorem are shown.  

\end{abstract}

\maketitle

\section*{Introduction}
\label{S:one}

A central theme in mathematical  analysis, classical or constructive, is the extension of a continuous function $f$, 
defined on a subspace $A$ of a space $X$ and satisfying some property $P$, to a continuous function $f^*$ defined 
on $X$ and satisfying the property $P$. A ``general extension theorem'' is a statement providing conditions on $X$ and 
$A$ such that such an extension of $f$ to $f^*$ is possible\footnote{Note that the general pattern of an extension theorem
is not to bestow extra properties upon $f$.}. As it is noted in~\cite{Du66}, p.~149, ``general extension theorems are 
rare and usually have interesting topological consequences''.

If $X$ is a normal (and Hausdorff) topological space, $A$ is a closed subspace of $X$, the continuous function on $A$
takes values in $\Real$, and $P(f)$ is ``$f$ is bounded'', the corresponding general extension theorem is Tietze's 
extension theorem, which is actually a characterisation of normality (see~\cite{Du66}, pp.~149-151) and its proof 
rests on Urysohn's lemma.
If $X$ is a metric space, hence normal, we get the restriction of the classical Tietze extension theorem to metric
spaces. If $X$ is a complete and separable metric space, Urysohn's lemma and Tietze's extension theorem are provable
in $\RCA_0$ (see~\cite{Si09}, pp.~90-91). Within Bishop-style constructivism, which is the framework of constructive
mathematics we work within, a Tietze extension theorem for metric spaces is proved in~\cite{BB85}, p.~120. According
to it,  if $X$ is a metric space, $A$ is a locally compact\footnote{In Bishop-style constructive analysis a locally 
compact metric space $X$ is a metric space such that every bounded subset of $X$ is included in a compact subset of $X$.} 
subset of $X$, and $f : A \to [a, b]$ is continuous i.e., it is uniformly continuous on every bounded subset 
of $A$ (see~\cite{BB85}, p.~110), there is an extension function $f^* : X \to [a, b]$ of $f$ on $X$, which is also uniformly 
continuous on every bounded subset of $X$. As it is customary in the constructive treatment of a classical extension theorem, 
the constructive theorem is less general, and sometimes the notions involved have to be understood differently. 
The benefit of the constructive approach though, is that the, usually hidden,  computational content of the classical 
theorem is revealed.  

If $X$ is a normed space, $A$ is a linear subspace of $X$ and the property $P(f)$ for the continuous function 
$f : A \to \Real$ is ``$f$ is linear and $||f|| = a$'', where $a \in [0, + \infty)$, the corresponding extension theorem
is (one version of) the Hahn-Banach theorem. Its proof is non-constructive, as it rests on the axiom of choice. 
Constructively, one needs to add some extra properties on $X$, in order to avoid the axiom of choice. E.g., in~\cite{Is89}
the normed space $X$ is also uniformly convex, complete and its norm is G\^ateaux differentiable. Since constructively 
the existence of the norm of the continuous functional is equivalent to the locatedness of its kernel 
(see~\cite{BV06}, pp.~53-54), in the formulation of the constructive Hahn-Banach threorem we need to suppose that
the norm $||f||$ of $f$ exists. Then the norm $||f^*||$ of the extension function $f^*$ exists too, and $||f^*|| = ||f||$. 
The computational content of the classical theorem is revealed in the constructively provable, \textit{approximate version}
of the Hahn-Banach theorem (see~\cite{BR87}, pp.~39-40). According to it, 
if $X$ is a separable normed space and $||f||$ exists, then for every $\epsilon > 0$ there is a continuous and linear
extension $f^*$ of $f$ with a norm $||f^*||$ satisfying 
$$||f^*|| \leq ||f|| + \epsilon.$$
Hence, in the approximate version of the Hahn-Banach theorem the numerical part of $P(f)$ is satisfied only approximately.

If $X$ is a metric space, $A$ is any subspace of $X$ and the property $P(f)$ for the (necessarily continuous) 
function $f : A \to \Real$ is ``$f$ is Lipschitz'', the corresponding extension theorem is the McShane-Whitney theorem, 
which appeared first\footnote{To determine metric spaces $X$ and $Y$, such that a similar extension 
theorem for $Y$-valued Lipschitz functions defined on a subset $A$ of $X$ holds, is a non-trivial problem under
active current study (see~\cite{BL00},~\cite{BB12} and~\cite{Tu14}).} in~\cite{Mc34} and~\cite{Wh34}. This 
theorem admits a classical proof that uses the axiom of choice, and it is very similar to the proof of the analytic
version of the Hahn-Banach theorem (see~\cite{Co12}, pp.~224-5). 

Following~\cite{We99}, pp.16-17, if $\sigma > 0$ and $f$ is $\sigma$-Lipschitz i.e.,
$\forall_{a, a' \in A}(|f(a) - f(a')| \leq \sigma d(a, a'))$,  
and if $x \in X \setminus A$, a $\sigma$-Lipschitz extension $g$ of $f$ on $A \cup \{x\}$ must satisfy
\[\forall_{a \in A}\big(|g(x) - g(a)| = |g(x) - f(a)| \leq \sigma d(x, a)\big),\]
which is equivalent to 
\[\forall_{a \in A}\big(f(a) - \sigma d(x, a) \leq g(x) \leq f(a) + \sigma d(x, a)\big).\]
Since $\forall_{a, a' \in A}\big(|f(a) - f(a')| \leq \sigma d(a, a') \leq \sigma (d(a, x) + d(x, a'))\big)$, 
we get
\[\forall_{a, a' \in A}\big(f(a) - \sigma  d(a, x) \leq f(a') + \sigma d(x, a')\big).\]
Hence, if we fix some $a \in A$ (the case $A = \emptyset$ is trivial), then for all $a' \in A$ we get
\[f(a) - \sigma  d(a, x) \leq f(a') + \sigma d(x, a'),\]
therefore, using the classical completeness property of real numbers, we get
\[f(a) - \sigma  d(x, a) \leq \inf \{f(a') + \sigma d(x, a') \mid a' \in A\}.\]
Since $a$ is an arbitrary element of $A$, we conclude that
\[s_x \equiv \sup \{f(a) - \sigma  d(x, a) \mid a \in A\} \leq \inf\{f(a) + \sigma d(x, a) \mid a \in A\} 
\equiv i_x.\]
If we define $g(x)$ to be any value in the interval $[s_x, i_x]$, we get the required extension of $f$ on 
$A \cup \{x\}$.
If we apply Zorn's lemma on the non-empty poset
\[\{(B, h) \mid B \supseteq A \ \& \ h : B \to \Real \ \mbox{is} \ \sigma \mbox{-Lipschitz} \ \& \  h_{|A} = f\},\]
where
\[(B_1, h_1) \leq (B_2, h_2) : \TOT B_1 \subseteq B_2 \ \& \ {h_2}_{|B_1} = h_1,\]
its maximal element is the required $\sigma$-Lipschitz extension of $f$. 

What McShane and Whitney independently observed though, was that if one defines $g^*(x) = s_x$ and
 $^*g(x) = i_x$, for all $x \in X$, the resulting functions $g^*$ and $^*g$ on $X$ are $\sigma$-Lipschitz 
 extensions of $f$, and, as expected by the above possible choice of $g(x)$ in $[s_x, i_x]$, every other 
 $\sigma$-Lipschitz extension of $f$ is between $g^*$ and $^*g$. In this way, the McShane-Whitney explicit
 construction of the 
extensions $g^*$ and $^*g$ of $f$ avoids Zorn's lemma.

If $X$ is a metric space and $A$ is a subset of $X$, we say that $(X, A)$ is a McShane-Whitney pair, if the
functions $g^*$ and $^*g$ are well-defined (Definition~\ref{def: MW}). Classically, all pairs $(X, A)$, where $A \neq 
\emptyset$, are McShane-Whitney
pairs. The definitions of $g^*$ and $^*g$ though, depend on the existence of the supremum and infimum of non-empty, 
bounded above and bounded below subsets of reals, respectively. Since the classical completeness property of reals 
implies constructively the principle of the excluded middle $(\PEM)$ (see~\cite{BV06}, p.~32), it is not a surprise 
that with a similar argument we can show that the statement ``every pair $(X, A)$ is a McShane-Whitney pair'' implies 
constructively the principle $\PEM$ (Remark~\ref{rem: pem2}).


In this paper we study the McShane-Whitney extension within Bishop's informal system of constructive
mathematics $\BISH$ (see~\cite{Bi67}, \cite{BB85},~\cite{BR87} and~\cite{BV06}). A formal system for $\BISH$ is 
Myhill's system $\CST$ with dependent choice (see~\cite{My75}), or Aczel's system of constructive set theory $\CZF$ 
with dependent choice (see~\cite{AR10}).

The main results included here reveal, in our view, the deep interactions between the McShane-Whitney theorem and 
the Hahn-Banach theorem. We can summarise our results as follows:

\begin{itemize}
\item If $X$ is a totally bounded metric space, we prove the uniform density of the Lipschitz functions 
on $X$ in the set of uniformly continuous functions on $X$ (Corollary 1.1). For that we prove 
Theorem~\ref{thm: density1} that has a formulation similar to the formulation of the McShane-Whitney 
theorem (Theorem~\ref{thm: MW}).
 \item We introduce the notion of a McShane-Whitney pair that reveals the constructive content of the classical
 McShane-Whitney extension in Theorem~\ref{thm: MW}. 
 \item We describe how properties of $f$, like boundedness, existence of the Lipschitz constant $L(f)$ of $f$,
 linearity, and convexity,  extend to $f^*$ and $^*f$ (subsections~\ref{s:metric} 
 and~\ref{s:normpairs}).
\item We provide conditions on $X$ and/or $Y$ so that $(X, A)$ is a McShane-Whitney pair, from the constructive
point of view (Proposition 2.1).
\item We define a local version of a McShane-Whitney pair (subsection~\ref{s:local}).
\item We prove a Lipschitz version of a fundamental corollary of the Hahn-Banach theorem 
(Theorem~\ref{thm: corhb2}).
\item We show how the McShane-Whitney pairs and the McShane-Whitney construction can be applied to 
H\"{o}lder functions and $\nu$-continuous functions, where $\nu$ is a modulus of continuity 
(subsection~\ref{s:hoelder}).
\item We prove the approximate McShane-Whitney theorem for Lipschitz functions defined on a finite-dimensional 
subspace of a normed space (Theorem~\ref{thm: aprox}). This is, in our view, the most technically interesting 
result of this paper.
 
\end{itemize}

First we prove some fundamental facts in the constructive theory of Lipschitz functions, which is quite
different from its classical counterpart.

\section{Lipschitz functions, constructively}\label{S:basic}

Constructive Lipschitz analysis is quite underdeveloped. Some related results are found scattered e.g.,
in~\cite{De69},~\cite{JP85},~\cite{Lo13},~\cite{BHP16}. For the sake of completeness, we include the definitions 
of some fundamental concepts.

\begin{definition}\label{def: real} Let $A \subseteq \Real$ and $b \in \Real$.
If $A$ is bounded above, we define 
$$b \geq A :\Leftrightarrow \forall_{a \in A}(b \geq a), $$
$$[A) := \{c \in \Real \mid c \geq A\},$$ 
$$l := \sup A :\Leftrightarrow l \geq A \ \& \ \forall_{\epsilon > 0}\exists_{a \in A}(a > l - \epsilon),$$ 
$$\lambda := \lub A :\Leftrightarrow \lambda \geq A  \ \& \ \forall_{b \in [A)}(b \geq \lambda).$$ 
If $A$ is bounded below, $b \leq A$, $(A]$,  $m := \inf A$, and $\mu := \glb A$ are defined in a dual way. 
\end{definition}

According to the classical completeness of $\Real$, if $A \subseteq \Real$ is inhabited and bounded above,
then $\sup A$ exists. Constructively, this statement implies the principle of the excluded middle $\PEM$ (i.e., 
the scheme $P \vee \neg P$) (see~\cite{BV06}, p.~32). A similar argument\footnote{This argument goes as follows. 
If $A := \{0\} \cup \{x \in \Real \mid x = 1 \ \& \ P\}$ and $\lambda := \lub A$ exists, then by the constructive 
version of the classical trichotomy of reals (Corollary (2.17) in~\cite{BB85}, p.~26) we have that $\lambda > 0$ 
or $\lambda < 1$. If $\lambda > 0$, we suppose $\neg P$, hence $A = \{0\}$, and $\lambda$ cannot be the least upper
bound of $\{0\}$. Hence we showed $\neg \neg P$. If $\lambda < 1$, we suppose $P$, hence $A = \{0, 1\}$ and we get 
the contradiction $1 \leq \lambda < 1$. Hence we showed $\neg P$.} shows that the corresponding existence of $\lub A$ 
implies the weak principle of the excluded middle $\WPEM$ (i.e., the scheme $\neg P \vee \neg \neg P$). It is immediate 
to show that if $\sup A$ exists, then $\lub A$ exists and they are equal, for every $A \subseteq \Real$. The converse
is not generally true (see~\cite{Ma83}, p.~27).

%
%
%

Note that Bishop used the terms $\lub A$ and $\sup A$ (and similarly the terms $\glb A$ and $\inf A$) without 
distinction. In~\cite{Ma83} Mandelkern gave reasons to keep both notions at work. In~\cite{Ma83}, pp.~24-25, 
he also gave a necessary and sufficient condition for the existence
of $\lub A$ and $\glb A$, similar to the constructive version of completeness given by Bishop and Bridges
in~\cite{BB85}, p.~37, and he proved the following remark: If $A \subseteq \Real$ is bounded and $\glb A$ exists,
then $\sup (A]$ exists and $\sup (A] = \glb A$, while if $\lub A$ exists, then $\inf [A)$ exists
and $\inf [A) = \lub A$.

\begin{definition}\label{def: ucont} If $X, Y$ are sets, we denote by $\Fii(X, Y)$ the set of functions of 
type $X \rightarrow Y$ and by $\Fii(X)$ the set of functions of type $X \rightarrow \Real$. If $a \in \Real$, 
then $\overline{a}_{X}$ 
denotes the constant map in $\Fii(X)$ with value $a$, and $\Const(X)$ is the set of constant maps on $X$. 
If $(X, d), (Y, \rho)$ are metric spaces, then $C_u(X, Y)$ denotes the set of uniformly continuous functions
from $X$ to $Y$, where $f \in C_u(X, Y)$, if there is a function $\omega_f : \Real^+ \to \Real^+$, a
\textit{modulus of uniform continuity} of $f$, such that for every $\epsilon > 0$
$$\forall_{x, y \in X}\big(d(x, y) \leq \omega_{f}(\epsilon) \Rightarrow \rho(f(x), f(y)) \leq \epsilon\big).$$
We denote by $C_u(X)$ the set $C_u(X, \Real)$, where $\Real$ is equipped with its standard metric. The metric
$d_{x_{0}}$ at the point $x_{0} \in X$ is defined by
$d_{x_{0}}(x) := d(x_{0}, x),$ for every $x \in X$, and $U_{0}(X) := \{d_{x_{0}} \mid x_{0} \in X\}.$ 
The set $X_{0}$ of the $d$-distinct pairs of $X$ is defined by 
$$X_{0} := \{(x, y) \in X \times X \mid d(x, y) > 0\}.$$

\end{definition}

The Lipschitz functions are defined constructively as in the classical setting. Throughout this section, $X$ is a
set equipped with a metric $d$, and $Y$ is a set equipped with a metric $\rho$.

\begin{definition}\label{def: Lip} 
The set of 
\textit{Lipschitz functions} from $X$ to $Y$ is defined by
$$\Lip(X, Y) := \bigcup _{\sigma \geq 0} \Lip(X, Y, \sigma),$$
$$\Lip(X, Y, \sigma) := \{f \in \mathbb{F}(X, Y) \mid \forall_{x, y \in X}(\rho(f(x), f(y)) \leq 
\sigma d(x, y))\}.$$
If $Y = \Real$, we write $\Lip(X)$ and $\Lip(X, \sigma)$, respectively.
\end{definition}

Clearly, $\Lip(X, Y) \subseteq C_u(X, Y)$. An element of $\LXY$ sends a bounded subset
of $X$ to a bounded subset of $Y$, which is not generally the case for elements of $C_u(X, Y)$; the
identity function $\id: \Nat \rightarrow \Real$, where $\Nat$ is equipped with the discrete metric, is 
in $\CN \setminus \LN$ and $\id(\Nat) = \Nat$ is unbounded in $\Real$. The following proposition is easy
to show.

\begin{proposition}\label{prp: Lip22} The set $\Lip(X)$
includes the sets $U_{0}(X)$, $\Const(X)$, and it is closed under addition and multiplication by reals.
If every element of $C_u(X)$ is a bounded function, then $\Lip(X)$ is 
closed under multiplication.
\end{proposition}

Next we show\footnote{This proof appeared first in~\cite{Pe16}.} constructively the uniform density of
$\Lip(X)$ in $C_u(X)$, in case $X$ is \textit{totally bounded}
i.e., if there is a finite $\epsilon$-approximation to $X$, for every $\epsilon > 0$ (see~\cite{BB85}, p.~94).
The formulation and the proof of the following theorem are analogous to the formulation and the proof of Theorem~\ref{thm: MW}.

%

\begin{theorem}\label{thm: density1} Let $(X, d)$ be totally bounded.
If $f \in C_{u}(X)$ and $\epsilon > 0$, there are $\sigma > 0$, $g^{*} \in \Lip(X, \sigma)$ and ${^{*}}g
\in \Lip(X, \sigma)$ such that the following hold:\\[1mm]
$(i) $ $f - \epsilon \leq g^{*} \leq f \leq {^{*}}g \leq f + \epsilon$.\\[1mm]
$(ii) $ For every $e \in \Lip(X, \sigma)$, $e \leq f \Rightarrow e \leq g^{*}$.\\[1mm]
$(iii) $ For every $e \in \Lip(X, \sigma)$, $f \leq e \Rightarrow {^{*}}g \leq e$. \\[1mm]
$(iv)$ The pair $(g^{*}$, ${^{*}}g)$ is the unique pair of functions in $\Lip(X, \sigma)$
satisfying conditions $(i)${-}$(iii)$. 
\end{theorem}

\begin{proof}
(i) Let $\omega_{f}$ be a modulus of uniform continuity of $f$, and $M_{f} > 0$ a bound of $f$. 
We define the functions $h_{x}: X \rightarrow \Real$ and $g^* : X \rightarrow \Real$ by
$$h_{x} := f + \sigma d_{x},$$
$$\sigma := \frac{2M_{f}}{\omega_{f}(\epsilon)} > 0,$$
$$g^*(x) := \inf \{h_{x}(y) \mid y \in X\} = \inf \{f(y) + \sigma d(x, y) \mid y \in X\},$$ 
for every $x \in X$. Note that $g^*(x)$ is well-defined, since $h_{x} \in C_{u}(X)$ and the infimum of 
$h_{x}$ exists, since $X$ is totally bounded (see~\cite{BB85}, p.94 and p.38).
First we show that $g \in \Lip(X, \sigma)$. If $x_{1}, x_{2}, y \in X$, the inequality 
$d(x_{1}, y) \leq d(x_{2}, y) + d(x_{1}, x_{2})$ implies that 
$f(y) + \sigma d(x_{1}, y) \leq (f(y) + \sigma d(x_{2}, y)) + \sigma d(x_{1}, x_{2})$, hence
$g^*(x_{1}) \leq (f(y) + \sigma d(x_{2}, y)) + \sigma d(x_{1}, x_{2})$, therefore
$g^*(x_{1}) \leq g^*(x_{2}) + \sigma d(x_{1}, x_{2})$, or 
$g^*(x_{1}) - g^*(x_{2}) \leq \sigma d(x_{1}, x_{2})$.
Starting with the inequality $d(x_{2}, y) \leq d(x_{1}, y) + d(x_{1}, x_{2})$ and working similarly we get that 
$g^*(x_{2}) - g^*(x_{1}) \leq \sigma d(x_{1}, x_{2})$, therefore $|g^*(x_{1}) - g^*(x_{2})| \leq \sigma d(x_{1}, x_{2})$.
Next we show that
$$\forall_{x \in X}\big(f(x) - \epsilon \leq g^*(x) \leq f(x)\big).$$
Since $f(x) = f(x) + \sigma d(x, x) = h_{x}(x) \geq 
\inf \{h_{x}(y) \mid y \in X\} = g^*(x)$, for every $x \in X$, we have that $g^* \leq f$. Next we show that
$\forall_{x \in X}(f(x) - \epsilon \leq g^*(x))$. For that we fix $x \in X$ and we show that
$\neg{(f(x) - \epsilon > g^*(x))}$.
Note that if $A \subseteq \Real, b \in \Real$, then\footnote{By the definition of $\inf A$ in~\cite{BB85}, p.37, 
we have that
$\forall_{\epsilon > 0}\exists_{a \in A}(a < \inf A + \epsilon)$, therefore if $b > \inf A$ and 
$\epsilon = b - \inf A > 0$ we get that $\exists_{a \in A}(a < \inf A + (b - \inf A) = b)$.}
$b > \inf A \Rightarrow \exists_{a \in A}(a < b)$. Therefore,
\begin{align*}
 f(x) - \epsilon > g^*(x) & \Leftrightarrow f(x) - \epsilon > \inf \{f(y) + \sigma d(x, y) \mid y \in X\}\\
& \Rightarrow \exists_{y \in X}(f(x) - \epsilon > f(y) + \sigma d(x, y)) \\
& \Leftrightarrow  \exists_{y \in X}(f(x) - f(y) > \epsilon + \sigma d(x, y)).
\end{align*}
For this $y$ we show that $d(x, y) \leq \omega_{f}(\epsilon)$. If $d(x, y) > \omega_{f}(\epsilon)$, we have that 
$$2M_{f} \geq f(x) + M_{f} \geq f(x) - f(y) > \epsilon + 2M_{f}\frac{d(x, y)}{\omega_{f}(\epsilon)} > \epsilon
+ 2M_{f} > 2M_{f},$$
which is a contradiction. Hence, by the uniform continuity of $f$ we get that $|f(x) - f(y)| \leq \epsilon$, therefore
the contradiction $\epsilon > \epsilon$ is reached, since
$$\epsilon \geq |f(x) - f(y)| \geq f(x) - f(y) > \epsilon + \sigma d(x, y) \geq \epsilon.$$
Next we define the functions $h_{x}^{*}: X \rightarrow \Real$ and 
${^{*}}g: X \rightarrow \Real$ by
$$h_{x}^{*} := f - \sigma d_{x},$$
$${^{*}}g(x) := \sup \{h_{x}^{*}(y) \mid y \in X\} = \sup \{f(y) - \sigma d(x, y) \mid y \in X\},$$ 
for every $x \in X$, and $\sigma = \frac{2M_{f}}{\omega_{f}(\epsilon)}$. Similarly\footnote{To show that 
$\neg{(g^{*}(x) > f(x) + \epsilon)}$ we just use the fact that if $A \subseteq \Real, b \in \Real$, then
$\sup A > b \Rightarrow \exists_{a \in A}(a > b)$. The function ${^{*}}g$ is mentioned in~\cite{web}, 
where non-constructive properties of the classical $(\Real, <)$ are used.}
to the proof for $g^*$, we get that ${^{*}}g \in \Lip(X, \sigma)$ and 
$$\forall_{x \in X}(f(x) \leq {^{*}}g(x) \leq f(x) + \epsilon).$$
(ii) Let $e \in \Lip(X, \sigma)$ such that $e \leq f$. If we fix some $x \in X$, then for every $y \in X$ we have that
$e(x) - e(y) \leq |e(x) - e(y)| \leq \sigma d(x, y)$, hence $e(x) \leq e(y) + \sigma d(x, y) \leq  
f(y) + \sigma d(x, y)$, therefore $e(x) \leq g^*(x)$.\\
(iii) Let $e \in \Lip(X, \sigma)$ such that $f \leq e$. If we fix some $x \in X$, then for every $y \in X$ 
we have that $e(y) - e(x) \leq |e(y) - e(x)| \leq \sigma d(x, y)$, hence
$f(y) - \sigma d(x, y) \leq  e(y) - \sigma d(x, y) \leq e(x)$, therefore ${^{*}}g(x) \leq e(x)$.\\
(iv) Let $h^{*}, {^{*}}h \in \Lip(X, \sigma)$ such that the following hold:\\[1mm]
(i)$^{'}$ $f - \epsilon \leq h^{*} \leq f \leq {^{*}}h \leq f + \epsilon$.\\
(ii)$^{'}$ For every $e \in \Lip(X, \sigma)$, $e \leq f \Rightarrow e \leq h^{*}$.\\
(iii)$^{'}$ For every $e \in \Lip(X, \sigma)$, $f \leq e \Rightarrow {^{*}}h \leq e$.\\[1mm]
Since $h^{*} \leq f$, we get by (ii) that $h^{*} \leq g^*$. Since $g^{*} \leq f$, we get by (ii)$^{'}$
that $g^{*} \leq h^*$. Using similarly (iii) and (iii)$^{'}$, we get ${^{*}}h = {^{*}}g$.
\end{proof}

From Theorem~\ref{thm: density1}(i) we get immediately the following corollary.

\begin{corollary}\label{crl: density2} 
If $X$ is totally bounded, then $\Lip(X)$ is uniformly dense in $C_{u}(X)$. 
 
\end{corollary}

%
%
%

The constructive behavior of Lipschitz functions is  different from the classical one. For example,
Lebesgue's theorem on the almost everywhere differentiability of a Lipschitz function $f : (a, b) \to \Real$,
and hence Rademacher's generalization (see~\cite{He05}, p.~18), do not hold constructively; Demuth has
shown in~\cite{De69} that in the recursive variety RUSS of BISH (see~\cite{BR87}) there is a Lipschitz 
function $f : [0, 1] \to \Real$,  which is nowhere differentiable. As we explain in the rest of this paper, 
there are important differences between the classical and the constructive theory of Lipschitz functions at 
an even more elementary level.

\begin{definition}\label{def: M0}
If $f \in \Fii(X, Y)$, we define the following sets:
$$\Lambda(f) := \{\sigma \geq 0 \mid \forall_{x, y \in X}(\rho(f(x), f(y)) \leq \sigma d(x, y))\},$$
$$\Xi(f) := \{\sigma \geq 0 \mid \forall_{x, y \in X}(\rho(f(x), f(y)) \geq \sigma d(x, y))\},$$
$$M_{0}(f) := \{\sigma_{x, y}(f) \mid (x, y) \in X_{0}\},$$
$$\sigma_{x, y}(f) := \frac{\rho(f(x), f(y))}{d(x, y)}.$$
\end{definition}

\begin{proposition}\label{prp: m}
 If $f \in \Fii(X, Y)$, then $\Lambda(f) = [M_{0}(f))$ and
 $\Xi(f) = (M_{0}(f)]$.
\end{proposition}

 \begin{proof}
 First we show that $\Lambda(f) = [M_{0}(f))$. If $b \geq M_{0}(f)$, we show that $b \in \Lambda(f)$
 i.e., $\forall_{x, y \in X}(\rho(f(x), f(y)) \leq b d(x, y)).$
 Since $b \geq 0$, let $x, y \in X$ such that 
 $\rho(f(x), f(y)) > b d(x, y)$.
 Suppose next that $d(x, y) > 0$, hence $\sigma_{x, y}(f) \in M_{0}(f)$, and since
 $b \geq M_{0}(f)$, we get that $\sigma_{x, y}(f) \leq b \Leftrightarrow \rho(f(x), f(y)) \leq b d(x, y)$ which 
 together with our initial hypothesis $b d(x, y) < \rho(f(x), f(y))$ leads to a contradiction.
 Hence, $d(x, y) \leq 0 \Leftrightarrow d(x, y) = 0 \Leftrightarrow x = y$, and our initial hypothesis
 is reduced to $0 > 0$. The converse implication, $\sigma \in \Lambda(f) \Rightarrow
 \sigma \geq M_{0}(f)$ follows immediately. The proof of the equality $\Xi(f) = (M_{0}(f)]$ is similar.
 \end{proof}

By the classical completeness of $\Real$, if $f \in \LXY$, then $\inf \Lambda(f)$ exists. Constructively, we don't 
have in 
general the existence of $\inf \Lambda(f)$. Actually, the statement ``$\inf \Lambda(f)$ exists, for every $f \in \Lip(A)$
and every $A \subseteq \Real$'' implies $\WPEM$. To see this, we work as follows. If $P$ is a syntactically well-formed 
proposition, let $A := \{0\} \cup \{x \in \Real \mid x = 1 \ \& \ P\}$. If $f : A \to \Real$ is defined by $f(a) := a$, 
for every $a \in A$, then $f \in \Lip(A, 1)$ and $1 \in \Lambda(f)$. If $m := \inf \Lambda(f)$, then by the constructive
version of the classical trichotomy of reals (Corollary (2.17) in~\cite{BB85}, p.~26) we have that $m > 0$ or $m < 1$.
If $m > 0$, we suppose $\neg P$, hence $A = \{0\}$. In this case $0 \in \Lambda(f)$, and we have that $m \leq 0 < m$,
which is a contradiction, hence we showed $\neg \neg P$. If $m < 1$, we suppose $P$, hence $A = \{0, 1\}$. By the
definition of $\inf \Lambda(f)$ we can find $\epsilon > 0$ and $\sigma \in \Lambda(f)$ such that
$\sigma < m + \epsilon < 1$. Hence, $1 = |0 - 1| \leq \sigma |0 - 1| = \sigma < 1$, which is a contradiction. 
Hence we showed $\neg P$.

Classically, if $\inf \Lambda(f)$ exists, then $\sup M_0(f)$ exists and it is equal to $\inf \Lambda(f)$.
Constructively, from the existence of $\inf \Lambda(f)$ we only get
the existence of $\lub M_0(f)$. If $\sup M_{0}(f)$ exists though, we can infer constructively the existence
of $\inf \Lambda(f)$.

\begin{proposition}\label{prp: infsup} Let $f \in \Lip(X, Y)$. \\
$(i)$ If $\sup M_{0}(f)$ exists, then $\inf \Lambda(f)$ exists and $\inf \Lambda(f)  
= \min \Lambda(f) = \sup M_{0}(f)$. \\
$(ii)$ If $\inf \Lambda(f)$ exists, then $\lub M_{0}(f)$ exists and
$\lub M_{0}(f) = \inf \Lambda(f)$.\\
$(iii)$ If $\lub M_{0}(f)$ exists, then $\inf \Lambda(f)$ exists and
$\inf \Lambda(f) = \lub M_{0}(f)$.

\end{proposition}

\begin{proof}
 (i) Since for $l = \sup M_{0}(f)$ we have that $l \geq M_{0}(f)$, by Proposition~\ref{prp: m}
 we get that $l \in \Lambda(f)$. Next we show that 
 $l \leq \Lambda(f)$. If $\sigma \in \Lambda(f)$, then $\sigma \geq M_{0}(f)$, therefore $\sigma \geq l$,
 since the supremum of a set is its least upper bound. Since $l \in \Lambda(f)$, if $\epsilon >0$, then
 $l < l + \epsilon$, hence $l = \inf \Lambda(f)$.\\
 (ii) If $v = \inf \Lambda(f)$, then $v \geq M_{0}(f)$, as we have explained in (i). Suppose next that 
 $b \geq M_{0}(f)$.
 We show that $b \geq v$, by supposing $b < v$ and reaching a contradiction. Since $b \geq M_{0}(f)$,
 by Proposition~\ref{prp: m} we have that $b \in \Lambda(f)$, and consequently $v \leq b < v$.\\
 (iii) By Proposition~\ref{prp: m} we have that $\Lambda(f) = [M_{0}(f))$, hence by Mandelkern's remark after
Definition~\ref{def: real} we get that $\lub M_{0}(f) = \inf [M_{0}(f)) = \inf \Lambda(f)$.
\end{proof}

\begin{definition}\label{def: L-norm} Let $f \in \Lip(X, Y)$. 
If $\sup M_{0}(f)$ exists, we call $L(f) := \sup M_{0}(f)$ the \textit{Lipschitz constant} of $f$. 
If $\lub M_{0}(f)$ exists, we call $L^*(f) := \lub M_{0}(f)$ the \textit{weak Lipschitz constant} of $f$.
\end{definition}

Constructively, $L(f)$ does not generally exist. Actually, the statement ``$L(f)$ exists, for every
$f \in \Lip(l_2)$'' implies the limited principle of omniscience (LPO), where LPO is a weak form of PEM, 
and $l_2 := \{(x_n) \in \mathbb{F} (\mathbb{N}, \Real) \mid \sum_{n \in \mathbb{N}}|x_n|^2 \in \Real \}$ is
equipped with the $2$-norm. This is explained as follows. A linear functional on a normed space $X$ is bounded,
if there is $\sigma > 0$ such that $|f(x)| \leq \sigma ||x||$, for every $x \in X$ (see~\cite{BV06}, p.~51). By
the linearity of $f$ we then get $f \in \Lip(X, \sigma)$. Since in this case it is easy to show that
$M_0 (f) = \{|f(x)| \mid ||x|| = 1\}$, the existence of $\sup M_0 (f)$ is equivalent to the existence of $||f||$.
Hence, the statement ``$L(f)$ exists, for every $f \in \Lip(l_2)$'' implies the statement ``every bounded linear
functional on $l_2$ is normable''. The last statement implies LPO (see~\cite{Is16}).

%
%
%
%


If $L(f)$ exists, and since $L(f) \geq M_{0}(f)$,
by Proposition~\ref{prp: m} we get 
$\forall_{x, y \in X}(\rho(f(x), f(y)) \leq L(f) d(x, y)),$
i.e., $f \in \Lip(X, Y, L(f)).$
If $L^*(f)$ exists, then proceeding similarly we get 
$\forall_{x, y \in X}(\rho(f(x), f(y)) \leq L^{*}(f) d(x, y))$ i.e., $f \in \Lip(X, Y, L^{*}(f)).$

\begin{proposition}\label{prp: beforeLBP1} Let $A \subseteq X$, $f \in \Lip(A, Y)$ and $g \in \Fii(X, Y)$ such that
$$g_{|A} = f \ \& \ \forall_{\sigma \geq 0}(f \in
\Lip(A, Y, \sigma) \Rightarrow g \in \Lip(X, Y, \sigma)).$$
$(i)$ If $L(f)$ exists, then $L(g)$ exists and $L(g) = L(f)$. \\
$(ii)$ If $L^*(f)$ exists, then $L^*(g)$ exists and $L^{*}(g) = L^{*}(f)$.
\end{proposition}

 \begin{proof} (i) If $L(f) = \sup M_{0}(f)$ we show that $L(f) = \sup M_{0}(g)$, where
 $M_{0}(f) = \{\sigma_{a, b}(f) \mid (a, b) \in A_{0}\},  M_{0}(g) = \{\sigma_{x, y}(g) \mid (x, y) \in X_{0}\},$
 and $M_{0}(f) \subseteq M_{0}(g)$. First we show that $L(f) \geq M_{0}(g)$. Since $L(f) \geq 0$ and
 $f \in \Lip(A, Y, L(f))$, we get by hypothesis that $g$ is also in $\Lip(X, Y, L(f))$, therefore
 $L(f) \geq M_{0}(g)$. If $\epsilon > 0$, there exists $\sigma_{a, b}(f) \in M_{0}(f)$
 such that $\sigma_{a, b}(f) > L(f) - \epsilon$, which implies that 
 $\sigma_{a, b}(f) = \sigma_{a, b}(g) > L(f) - \epsilon$, and $\sigma_{a, b}(g) \in M_{0}(g)$, 
 since $A_{0} \subseteq X_{0}$. \\
 (ii) If $L^{*}(f) = \lub M_{0}(f)$ we show that $L^{*}(f) = \lub M_{0}(g)$. 
 First we show that $L^{*}(f) \geq M_{0}(g)$. As in (i), since $L^{*}(f) \geq 0$ and
 $f \in \Lip(A, Y, L^{*}(f))$, we get by hypothesis that $g$ is also in $\Lip(X, Y, L^{*}(f))$, therefore
 $L^{*}(f) \geq M_{0}(g)$. If $m \geq M_{0}(g)$, then $m \geq M_{0}(f)$, since $M_{0}(g) \subseteq M_{0}(g)$,
 therefore $m \geq L^{*}(f)$.
 \end{proof}

Note that if $f \in \Lip(A, Y), g \in \Lip(X, Y)$ such that $L(f), L(g)$ exist and $L(f) = L(g)$, then
it is immediate to see that 
$\forall_{\sigma \geq 0}(f \in \Lip(A, Y, \sigma) \Rightarrow g \in \Lip(X, Y, \sigma))$.
Next follows the Lipschitz-version of the extendability of a uniformly continuous function defined 
on a dense subset of a metric space with values in a complete metric space. 
Its expected proof, which is based on the proof of Lemma 3.7 in~\cite{BB85}, pp.~91-2, is omitted. 

\begin{proposition}\label{prp: LBP1} Let $D \subseteq X$ be dense in $X$, $Y$ complete,
and $f \in \Lip(D, Y)$.\\[1mm]
$(i)$ $\exists_{!g \in \Fii(X, Y)}\big(g_{|A} = f \ \& \ \forall_{\sigma \geq 0}(f \in
\Lip(D, Y, \sigma) \Rightarrow g \in \Lip(X, Y, \sigma))\big)$.\\[1mm]
$(ii)$ If $L(f)$ exists and $g$ is as in $($i$)$, then $L(g)$ exists and $L(g) = L(f)$.\\[1mm]
$(iii)$ If $L^*(f)$ exists and $g$ is as in $($i$)$, then $L^*(g)$ exists and $L^{*}(g) = L^{*}(f)$.
\end{proposition}

\section{McShane-Whitney pairs}\label{S:MW}

\subsection{Metric spaces}\label{s:metric}

Throughout this section $X$ is a set equipped with a metric $d$, and $Y$ is a set equipped with a metric $\rho$.

\begin{definition}\label{def: MW}
Let $A \subseteq X$. We call $(X, A)$ a \textit{McShane-Whitney pair}, or simply an \textit{MW-pair}, if for every
$\sigma > 0$ and $g \in \Lip(A, \sigma)$ the functions $g^{*}, {^{*}}g : X \rightarrow \Real$ are well-defined,
where for every $x \in X$
$$g^{*}(x) := \sup \{g(a) - \sigma d(x, a) \mid a \in A\},$$
$${^{*}}g(x) := \inf \{g(a) + \sigma d(x, a) \mid a \in A\}.$$
 \end{definition}

Note that $(X, A)$ is a McShane-Whitney pair if and only if for every $g \in \Lip(A, 1)$ the function $g^{*}$ is 
well-defined, since we can work with $\frac{1}{\sigma}h$, for an arbitrary $\sigma > 0$ and $h \in \Lip(A, \sigma)$,
and use the equality ${^{*}}g = -[(-g)^*]$ for the extension ${^{*}}g$. 
The classical statement  ``every pair $(X, A)$ is a McShane-Whitney pair'' 
cannot be accepted constructively.

\begin{remark}\label{rem: pem2}
The statement ``$(\Real, A)$ is a McShane-Whitney pair, for every non-empty $A \subseteq \Real$'' implies $\PEM$.
\end{remark}

\begin{proof}
If $P$ is a syntactically well-formed proposition, let $A := \{0\} \cup \{x \in \Real \mid x = 1 \ \& \ P\}$. If
$g : A \to \Real$ is defined by $g(a) := a$, for every $a \in A$, then $g \in \Lip(A, 1)$. If $x := 2$, we suppose that
$$g^* (2) := \sup \{a - |2 - a| \mid a \in A\} \in \Real.$$
By the constructive version of the classical trichotomy of reals (Corollary (2.17) in~\cite{BB85}, p.~26) we 
have that $g^*(2) > -\frac{3}{2}$ or $g^*(2) < -1$. In the first case, by the definition of $\sup A$ there 
is $a \in A$ such that $a - |2 - a| >  -\frac{3}{2}$. If $a = 0$, then $0 - |2 - 0| = -2 < - \frac{3}{2}$, 
hence there must be $a \in \{x \in \Real \mid x = 1 \ \& \ P\}$ that satisfies this property. This 
implies $P$ and $g^*(2) = 0 > -\frac{3}{2}$. If $g^*(2) < -1$, we suppose $P$, hence $A = \{0, 1\}$. 
Since then $g^*(2) = 0$, we get the contradiction $0 = g^*(2) < -1$ i.e., we showed $\neg P$.
\end{proof}

The next theorem reveals the computational content of the classical McShane-Whitney theorem.

\begin{theorem}[McShane-Whitney]\label{thm: MW} If $(X, A)$ is an MW-pair and $g \in \Lip(A, \sigma)$, 
the following hold:\\[1mm]
$(i)$ $g^{*}, {^{*}}g \in \Lip(X, \sigma)$.\\[1mm]
$(ii)$ ${g^{*}}_{|A} = ({^{*}}g)_{|A} = g$.\\[1mm]
$(iii)$ $\forall_{f \in \Lip(X, \sigma)}\big(f_{|A} = g \Rightarrow g^{*} \leq f \leq {^{*}}g\big)$.\\[1mm]
$(iv)$ The pair $(g^{*}$, ${^{*}}g)$ is the unique pair of functions satisfying conditions $(i)${-}$(iii)$. 
 \end{theorem}

\begin{proof}
(i) Let $x_{1}, x_{2} \in X$ and $a \in A$. Then $d(x_{1}, a) \leq d(x_{2}, a) + d(x_{2}, x_{1})$ and
$\sigma d(x_{1}, a) \leq \sigma d(x_{2}, a) + \sigma d(x_{1}, x_{2})$, therefore
\begin{align*}
 & g(a) + \sigma d(x_{1}, a) \leq (g(a) + \sigma d(x_{2}, a)) + \sigma d(x_{1}, x_{2}) \Rightarrow \\
 & {^{*}g}(x_{1}) \leq (g(a) + \sigma d(x_{2}, a)) + \sigma d(x_{1}, x_{2}) \Rightarrow  \\
& {^{*}g}(x_{1}) - \sigma d(x_{1}, x_{2}) \leq g(a) + \sigma d(x_{2}, a) \Rightarrow \\
 & {^{*}g}(x_{1}) - \sigma d(x_{1}, x_{2}) \leq  {^{*}g}(x_{2}) \Rightarrow \\
 & {^{*}g}(x_{1}) - {^{*}g}(x_{2}) \leq \sigma d(x_{1}, x_{2}).
\end{align*}
If we start from the inequality $d(x_{2}, a) \leq d(x_{1}, a) + d(x_{2}, x_{1})$ and work as above, we get
${^{*}}g(x_{2}) - {^{*}g}(x_{1}) \leq \sigma d(x_{1}, x_{2})$, therefore 
$|{^{*}g}(x_{1}) - {^{*}g}(x_{2})| \leq \sigma d(x_{1}, x_{2})$.
Working similarly we get that $g^{*}$ is an extension of $g$ which is in $\Lip(X, \sigma)$.\\
(ii) We show that
${^{*}g}$ extends $g$. If $a_{0} \in A$, then ${^{*}g}(a_{0}) = \inf \{g(a) + 
\sigma d(a_{0}, a) \mid a \in A\} \leq g(a_{0}) + \sigma d(a_{0}, a_{0}) = g(a_{0})$. If $a \in A$, 
then $g(a_{0}) - g(a) \leq |g(a_{0}) - g(a)|
\leq \sigma d(a_{0}, a)$, hence $g(a) + \sigma d(a_{0}, a) \geq g(a_{0})$. Since $a$ is arbitrary,
${^{*}g}(a_{0}) \geq g(a_{0})$. \\
(iii) If $f \in \Lip(X, \sigma)$ such that $f_{|A} = g$, $x \in X$ and $a \in A$ we have that
\begin{align*}
 & f(x) - g(a) = f(x) - f(a) \leq |f(x) - f(a)| \leq \sigma d(x, a) \Rightarrow \\
 & f(x) \leq g(a) + \sigma d(a, x) \Rightarrow \\
 & f(x) \leq {^{*}g}(x), \\
 & g(a) - f(x) = f(a) - f(x) \leq |f(a) - f(x)| \leq \sigma d(x, a) \Rightarrow \\
 & g(a) - \sigma d(a, x) \leq f(x) \Rightarrow \\
 & g^{*}(x) \leq f(x). 
\end{align*}
(iv) Let $h^{*}, {^{*}}h$ satisfy conditions (i)-(iii). Since 
  ${h^{*}}_{|A} = ({^{*}}h)_{|A} = g$, we have that $g^{*} \leq h^{*} \leq {^{*}}g$ and $g^{*} 
  \leq {^{*}}h \leq {^{*}}g$.  Since ${g^{*}}_{|A} = ({^{*}}g)_{|A} = g$, we have that
 $h^{*} \leq g^{*} \leq {^{*}}h$ and $h^{*} \leq {^{*}}g \leq {^{*}}h$, hence 
 $h^{*} = g^{*}$ and ${^{*}}h = {^{*}}g$.  
\end{proof}

Since in the previous proof we only used the ``least upper bound'' property of $\sup$ and the ``greatest lower bound'' 
property of $\inf$, we could have defined $g^{*}$ and ${^{*}}g$ through $\lub$ and $\glb$, respectively.
Next we recall some fundamental definitions in the constructive theory of metric spaces.


\begin{definition}\label{def: compact}
A space $X$ is \textit{compact}, if it is complete and totally bounded, while it is 
\textit{locally} compact (totally bounded), if 
every bounded subset of $X$ is included in a compact (totally bounded) subset of $X$. 
A subset $A$ of $X$ is \textit{located}, if $d(x, A) := \inf \Delta (x, A)$, where $\Delta(x, A)
:= \{d(x, a) \mid a \in A\}$, exists, for every $x \in X$.
\end{definition}

\begin{proposition}\label{prp: examples} The following sets $X$ and $A$ form MW-pairs.\\[1mm]
$(i)$  $A$ is a totally bounded subset of $X$.\\[1mm]
$(ii)$ $X$ is totally bounded and $A$ is located.\\[1mm]
$(iii)$ $X$ is locally compact $($totally bounded$)$ and $A$ is bounded and located.
\end{proposition}

\begin{proof}
 (i) If $\sigma > 0$, $g \in \Lip(A)$ and $x \in X$, then $g + \sigma d_{x}, g - \sigma d_x
 \in \Lip(A) \subseteq C_u(A)$, hence $g^{*}(x)$ and ${^{*}}g(x)$ are well defined,
 since $A$ is totally bounded (see~\cite{BB85}, Corollary 4.3, p.94).\\
 (ii) A located subset of $X$ is also totally bounded (see~\cite{BB85}, p.95), and
 we use (i).\\
 (iii) If $A$ is bounded and located, 
 there is compact (totally bounded) $K \subseteq X$ with $A \subseteq K$. 
 Since $A$ is located in $X$, it is located in $K$, hence $A$ is totally bounded, and we use (i).\\
\end{proof}

\begin{proposition}\label{prp: infsup2}
 Let $(X, A)$ be an MW-pair and $g \in \Lip(A, \sigma)$. The following hold.\\[1mm]
$(i)$ The set $A$ is located.\\[1mm]
$(ii)$ If $M \in \Real$ such that $\forall_{a \in A}(g(a) \leq M)$, then  $\forall_{x \in X}(g^*(x) \leq M)$.\\[1mm]
$(iii)$ If $m \in \Real$ such that $\forall_{a \in A}(g(a) \geq m)$, then  $\forall_{x \in X}({^{*}}g(x) \geq m)$.\\[1mm]
$(iv)$ If $\inf g$ and $\sup g$ exist, then $\inf {^{*}g}$, $\sup g^{*}$ exist and the following equalities hold
 $$\inf_{x \in X} {^{*}g} = \inf_{a \in A} g \ \ \& \ \ \sup_{x \in X} {g^{*}} = \sup_{a \in A} g.$$
\end{proposition}

\begin{proof}(i) Let $r \in \Real$ and $\sigma > 0$. Since $\overline{r}_{A} \in \Lip(A, \sigma)$, by hypothesis 
${^{*}}\overline{r}_{A}$ is well-defined, where ${^{*}}\overline{r}_{A}(x) = \inf \{r + \sigma d(x, a) 
\mid a \in A\}$, for every $x \in X$.
If $x \in X$ and $a \in A$, then $d(x, a) = \frac{1}{\sigma}(r + \sigma d(x, a) - r)$, 
and $\Delta(x, A) = \{\frac{1}{\sigma}(r + \sigma d(x, a) - r) \mid a \in A\}$.
Hence
\begin{align*}
d(x, A) & = \inf \{\frac{1}{\sigma}(r + \sigma d(x, a) - r) \mid a \in A\} \\
& = \frac{1}{\sigma} ( \inf \{r + \sigma d(x, a) \mid a \in A\} - r) \\
& = \frac{1}{\sigma}({^{*}}\overline{r}_{A}(x)  - r).
\end{align*}
(ii) Let $x \in X$. Since $g(a) \leq M$, $g(a) - \sigma d(x, a) \leq M - \sigma d(x, a) \leq M$. Hence
$g^*(x) \leq M$. \\
(iii) Let $x \in X$. Since $g(a) \geq M$, $g(a) + \sigma d(x, a) \geq M$. Hence
${^{*}}g \leq M$. \\
(iv) We show that $m := \inf \{g(a) \mid a \in A\}$ satisfies the properties of $\inf \{{^{*}}g(x) \mid x \in X\}$.
 It suffices to show that $m \leq {^{*}}g(X)$, since the other definitional condition of $\inf$ follows
 immediately;
 if $\epsilon > 0$, then there exists $a \in A \subseteq X$ such that $g(a) = {^{*}}g(a) < m + \epsilon$.
 If $x \in A$, then $m \leq g(x) = {^{*}}g(x)$, since $m = \inf g$. Since $A$ is located, the set 
 $$-A := \{x \in X \mid d(x, A) > 0\}$$
 is well-defined. If $x \in -A$, then $d(x, a) \geq d(x, A) > 0$, 
 for every $a \in A$. Hence
 \begin{align*}
  & g(a) + \sigma d(x, a)  > g(a) \geq \inf_{a \in A} g \\
  & \Rightarrow 
   \inf_{a \in A}(g(a) + \sigma d(x, a)) \geq \inf_{a \in A} g\\
   & \Leftrightarrow 
   {^{*}g}(x) \geq m.
 \end{align*}
Since $A$ is located, the set $A \cup (-A)$ is dense in $X$ (see~\cite{BB85}, p.88). If $x \in X$, there
is some sequence
$(d_{n})_{n \in \Nat} \subseteq A \cup (-A)$ such that $d_{n} \stackrel{n} \rightarrow x$. By the 
continuity of
${^{*}}g$ we have that ${^{*}}g(d_{n}) \stackrel{n} \rightarrow {^{*}}g(x)$. Suppose that ${^{*}}g(x) < m$.
Since 
${^{*}}g(d_{n}) \geq m$, for every $n \in \Nat$, if
$\epsilon := (m - {^{*}}g(x)) > 0$, there is some $n_{0}$ such that for every $n \geq n_{0}$ we have that 
$$|{^{*}}g(d_{n}) - {^{*}}g(x)| = {^{*}}g(d_{n}) - {^{*}}g(x) < m - {^{*}}g(x) \Leftrightarrow
{^{*}}g(d_{n}) < m,$$
 which is a contradiction. Hence ${^{*}}g(x) \geq m$.
For the existence of $\sup g^{*}$ and the equality $\sup_{x \in X} {g^{*}} = \sup_{a \in A} g$ we work
similarly. 
\end{proof}

If $X$ is a normed space and $x_{0} \in X$, then it is not generally the case that the set
$\Real x_{0} := \{\lambda x_{0} \mid \lambda \in \Real\}$ is a located subset of $X$. If $X = \Real$, the 
locatedness of $\Real x_0$, for every $x_0 \in \Real$, is 
equivalent to LPO (see~\cite{BB85}, p.~122). 
If for every $x_{0} \in \Real$ we have that $(\Real, \Real x_{0})$ is a McShane-Whitney pair,
then by Proposition~\ref{prp: infsup2}(i) every set of the form $\Real x_{0}$ is located, which implies LPO.

The next proposition expresses the ``step-invariance'' of the McShane-Whitney extension. If 
$A \subseteq B \subseteq X$ such that $(X, A), (X, B)$ and $(B, A)$ are  McShane-Whitney pairs and $g \in \Lip(A)$,
then $g^{*_{X}}$ is the $(A${-}$X)$ extension of $g$, $g^{*_{B}*_{X}}$ is the $(B${-}$X)$ extension of the
$(A${-}$B)$ extension $g^{*_{B}}$ of $g$, and similarly for $^{*_{X}}g$ and ${^{*_{X}*_{B}}}g$.

\begin{proposition}\label{prp: stepmk}
 If $A \subseteq B \subseteq X$ such that $(X, A), (X, B)$, $(B, A)$ are 
 MW-pairs and $g \in \Lip(A, \sigma)$, for some $\sigma > 0$, then
 $$g^{*_{X}} = g^{*_{B}*_{X}}, \ \ \ ^{*_{X}}g = {^{*_{X}*_{B}}}g.$$
\end{proposition}

\begin{proof} We show only the first equality and for the second we work similarly. By definition
$g^{*_{B}}: B \rightarrow \Real \in \Lip(B, \sigma)$ and $g^{*_{B}}(b) = \sup\{g(a) - \sigma d(b, a) 
\mid a \in A\}$, for every
$b \in B$. Moreover, $g^{*_{B}*_{X}}: X \rightarrow \Real \in \Lip(X, \sigma)$ and 
$$g^{*_{B}*_{X}}(x) = \sup\{g^{*_{B}}(b) - \sigma d(x, b) \mid b \in B\},$$
 for every $x \in X$. 
For the $(A${-}$X)$ extension of
$g$ we have that $g^{*_{X}}: X \rightarrow \Real \in \Lip(X, \sigma)$ and 
$g^{*_{X}}(x) = \sup\{g(a) - \sigma d(x, a) \mid a \in A\}$, for every $x \in X$. 
Since $(g^{*_{B}*_{X}})_{|B} = g^{*_{B}}$, we have that  $(g^{*_{B}*_{X}})_{|A} = (g^{*_{B}})_{|A} = g$.
Therefore
$g^{*_{X}} \leq g^{*_{B}*_{X}} \leq {^{*_{X}}}g,$ and 
$(g^{*_{X}})_{|B} \leq (g^{*_{B}*_{X}})_{|B} = g^{*_{B}} \leq ({^{*_{X}}}g)_{|B}.$ 
Since $(g^{*_{X}})_{|A} = g$, we get that $((g^{*_{X}})_{|B})_{|A} = g$, therefore
$g^{*_{B}} \leq ((g^{*_{X}})_{|B}) \leq {^{*_{B}}}g.$
Since $(g^{*_{X}})_{|B} \leq g^{*_{B}}$ and $g^{*_{B}} \leq (g^{*_{X}})_{|B}$, we get
$(g^{*_{X}})_{|B} = g^{*_{B}}$. Hence
$g^{*_{B}*_{X}} \leq g^{*_{X}} \leq {^{*_{X}*_{B}}}g$
i.e., we have shown both $g^{*_{X}} \leq g^{*_{B}*_{X}}$ and $g^{*_{B}*_{X}} \leq g^{*_{X}}$.
\end{proof}

\begin{proposition}\label{prp: sameconstant}
If $(X, A)$ is an MW-pair and $g \in \Lip(A)$ such that $L(g)$ exists, the following hold.\\[1mm]
$(i)$ $g \in \Lip(A, L(g))$.\\[1mm]
$(ii)$ If $f$ is an $L(g)$-Lipschitz extension of $g$, then $L(f)$ exists and $L(f) = L(g)$.\\[1mm]
$(iii)$ $L({^{*}}g), L(g^{*})$ exist and  $L({^{*}}g) =  L(g) = L(g^{*}).$
\end{proposition}

\begin{proof} 
(i) See our comment before the Proposition~\ref{prp: beforeLBP1}.\\
(ii) Since $f \in \Lip(X, L(g))$, we get $L(g) \in \Lambda(f)$ and  
$L(g) \geq M_{0}(f)$. Let $\epsilon > 0$.
Since $L(g) = \sup M_{0}(g)$, there exists 
$(a, b) \in A_{0} \subseteq X_{0}$
such that $\sigma_{a, b}(g) > L(g) - \epsilon$. Since $f$ extends $g$,
$\sigma_{a, b}(g) = \sigma_{a, b}(f)$.\\
(iii) By definition ${^{*}}g, g^{*} \in \Lip(X, L(g))$ and they extend $g$. Hence we use (ii).
\end{proof}

\begin{proposition}\label{prp: operations}
 Let $(X, A)$ be an MW-pair, $g_{1} \in \Lip(A, \sigma_{1}), g_{2} \in \Lip(A, \sigma_{2})$ and
 $g \in \Lip(A, \sigma)$, for some $\sigma_{1}, \sigma_{2}, \sigma > 0$. The following hold.\\[1mm]
 $(i)$ $(g_{1} + g_{2})^{*} \leq g_{1}^{*} + g_{2}^{*}$ and ${^{*}}(g_{1} + g_{2}) \geq {^{*}}g_{1} 
 + {^{*}}g_{2}$.\\[1mm]
$(ii)$ If $\lambda > 0$, then $(\lambda g)^{*} = \lambda g^{*}$ and ${^{*}}(\lambda g) = \lambda {^{*}}g$.\\[1mm]
$(iii)$ If $\lambda < 0$, then $(\lambda g)^{*} = \lambda {^{*}g}$ and ${^{*}}(\lambda g) = \lambda g^{*}$.
\end{proposition}

\begin{proof}In each case we show only one of the two facts.\\
(i) We have that $g_{1} + g_{2} \in \Lip(A, \sigma_{1} + \sigma_{2})$ and
\begin{align*}
 (g_{1} + g_{2})^{*}(x) & = \sup \{g_{1}(a) + g_{2}(a) - (\sigma_{1} + \sigma_{2})d(x, a) \mid a \in A\} \\
 & = \sup \{(g_{1}(a) - \sigma_{1}d(x, a)) + (g_{2}(a) -  \sigma_{2}d(x, a)) \mid a \in A\} \\
 & \leq \sup \{g_{1}(a) - \sigma_{1}d(x, a) \mid a \in A\} + \sup \{g_{2}(a) -  \sigma_{2}d(x, a) \mid a \in A\}\\
 & = g_{1}^{*}(x) + g_{2}^{*}(x).
\end{align*}
(ii) If $\lambda \in \Real$, then $\lambda g \in \Lip(A, |\lambda|\sigma)$ and if $\lambda > 0$, then
\begin{align*}
 (\lambda g)^{*}(x) & = \sup \{\lambda g(a) - |\lambda|\sigma d(x, a) \mid a \in A\}\\
 & = \lambda \sup \{g(a) - \sigma d(x, a) \mid a \in A\}\\
 & = \lambda g^{*}(x).
\end{align*}
(iii) If $\lambda < 0$, then 
 \begin{align*}
 (\lambda g)^{*}(x) & = \sup \{\lambda g(a) - |\lambda|\sigma d(x, a) \mid a \in A\}\\
 & = \sup \{\lambda g(a) - (- \lambda)\sigma d(x, a) \mid a \in A\}\\
 & = \sup \{\lambda (g(a) + \sigma d(x, a)) \mid a \in A\}\\
 & = \lambda \inf \{g(a) + \sigma d(x, a) \mid a \in A\}\\
 & = \lambda {^{*}}g(x).
   \tag*{\qedhere}
\end{align*}
\end{proof}

\subsection{Local McShane-Whitney pairs}\label{s:local}

Next we define a local notion of an MW-pair. 

\begin{definition}\label{def: locally}The pair
$(X, A)$ is a \textit{local} McShane-Whitney pair, or simply a local MW-pair, if for every bounded
$B \subseteq A$ there is  $A{'} \subseteq A$ such that $B \subseteq A{'}$ and $(X, A{'})$ is a MW-pair.

\end{definition}

Clearly, an MW-pair is a local MW-pair. The next lemma is shown in~\cite{BR87}, p.~33 and in~\cite{BV06},
p.~46. If $A, B \subseteq X$, 
we use Sambin's notation $A \between B$ to denote that the intersection $A \cap B$ is inhabited.

\begin{lemma}\label{lem: BV} If $A$ is a located subset of a metric space $X$ and $T$ is a totally bounded
subset of $X$ such that $T \between A$, there is totally bounded $S \subseteq X$ such that  
$T \cap A \subseteq S \subseteq A$. 
\end{lemma}

The next proposition is the ``local'' version of Proposition~\ref{prp: examples}(i)-(ii).

\begin{proposition} If $A \subseteq X$, the following hold.\\[1mm]
$(i)$ If $X$ is locally totally bounded and $A$ is located, then $(X, A)$ is
a local MW-pair.\\[1mm]
$(ii)$ If $A$ is a locally totally bounded subset of $X$, then $(X, A)$ is a local MW-pair.
\end{proposition}

\begin{proof} 
(i) Let $B \subseteq A$ bounded. By the definition of local total boundedness there is a totally bounded 
$T \subseteq X$ such that $B \subseteq T$. Since $B \subseteq T \cap A$, and since by the definition of
a bounded set, $B$ is inhabited, we get $A \between T$. By Lemma~\ref{lem: BV} there is a totally bounded
$S \subseteq X$ such that 
$B \subseteq T \cap A \subseteq S \subseteq A$. Since $S$ is totally bounded, by Proposition~\ref{prp: examples}(i)
we have that $(X, S)$ is an MW-pair.\\
(ii) Let $B \subseteq A$ be bounded. By the definition of local total boundedness there is a totally bounded
$T \subseteq A$ such that $B \subseteq T$. As in case (i), we get that $(X, T)$ is an MW-pair.
\end{proof}

\subsection{H\"{o}lder continuous functions of order $\alpha$ and $\nu$-continuous functions}\label{s:hoelder}

As expected, the McShane-Whitney extension works for the standard generalizations of Lipschitz functions.

\begin{definition}\label{def: hoelder}
If $\sigma \geq 0$ and $\alpha \in (0, 1]$, the \textit{H\"{o}lder continuous functions of 
order $\alpha$} are defined as follows:
$$\Hoel(X, Y, \alpha, \sigma) := \{f \in \mathbb{F}(X, Y) \mid \forall_{x, y \in X}(\rho(f(x), f(y)) 
\leq \sigma d(x, y)^{\alpha})\},$$
$$\Hoel(X, Y, \alpha) := \bigcup _{\sigma \geq 0} \Hoel(X, Y, \alpha, \sigma).$$
If $Y = \Real$ is equipped with its standard metric, we write $\Hoel(X, \alpha, \sigma)$ and $\Hoel(X, \alpha)$,
respectively.
\end{definition}


Clearly, $\Hoel(X, Y, 1, \sigma) = \Lip(X, Y, \sigma)$ and $\Hoel(X, Y, 1) = \Lip(X, Y)$.
One can define MW-pairs with respect to H\"{o}lder continuous functions of order $\alpha$, and show 
a McShane-Whitney extension theorem for them. We give some more details of these two steps only 
for the following generalization of H\"{o}lder continuity. The next definition is a reformulation of the 
definition found in the formulation of Theorem 5.6 in~\cite{BB85}, p.~102.


\begin{definition}\label{def: lambda} A \textit{modulus of continuity} is a function 
$\nu : [0, + \infty) \rightarrow [0, + \infty)$ satisfying the following conditions:\\[1mm]
(i) $\nu(0) = 0$.\\[1mm]
(ii) $\forall_{x, y \in [0, + \infty)}\big(\nu(x + y) \leq \nu(x) + \nu(y)\big)$.\\[1mm]
(iii) It is strictly increasing i.e., $\forall_{s, t \in [0, + \infty)}\big(s < t \Rightarrow \nu(s) < \nu(t)\big)$.\\[1mm]
(iv) It is uniformly continuous on every bounded subset of $[0, + \infty)$.\\[1mm]
If $\nu$ is a modulus of continuity the set of $\nu$-\textit{continuous} functions from $X$ to $Y$, and from 
$X$ to $\Real$, are defined, respectively, as follows:
$$S(X, Y, \nu) := \{f \in C_u(X, Y) \mid \forall_{x, y \in X}(\rho(f(x), f(y))  \leq \nu(d(x, y))\},$$
$$S(X, \nu) := \{f \in C_u(X) \mid \forall_{x, y \in X}(|f(x) - f(y)|  \leq \nu(d(x, y))\},$$
\end{definition}
Clearly, if $\nu_{1}(t) = \sigma t$, $\nu_{2}(t) = \sigma t^{\alpha}$, for every $t \in [0, + \infty)$, 
then 
$$S(X, \nu_{1}) = \Lip(X, \sigma),$$ 
$$S(X, \nu_{2}) = \Hoel(X, \alpha, \sigma).$$
Notice that since $a \in (0, 1]$, we have that $\nu_{2}$ is concave, hence subadditive. 
It is immediate to show that if $\nu$ is a modulus of continuity, then
$$\forall_{s, t \in [0, + \infty)}\big(s \leq t \Rightarrow \nu(s) \leq \nu(t)\big).$$

%

\begin{definition}\label{def: lMWpairs}
Let $A$ be a subset of $X$. We call $(X, A)$ a $\nu$\textit{-McShane-Whitney pair}, 
or simply a $\nu${-}MW-pair, if for every
$g \in S(A, \nu)$ the functions $g^{*}_{\nu}, {^{*}}g_{\nu} : X \rightarrow \Real$ are well-defined, 
where for every $x \in X$
$$g^{*}_{\nu}(x) := \sup \{g(a) - \nu(d(x, a)) \mid a \in A\},$$
$${^{*}}g_{\nu}(x) := \inf \{g(a) + \nu(d(x, a)) \mid a \in A\}.$$
\end{definition}


In order to get examples of $\nu${-}MW-pairs, one proceeds as in Proposition~\ref{prp: examples}(i)-(iii). If, for example,
$A$ is totally bounded, then, since $d_{x}$ is uniformly continuous, $d_{x}(A)$ is 
a bounded subset of $[0, + \infty)$ and $\nu \circ d_{x} = \nu_{|d_{x}(A)} \circ d_{x}$ 
is a uniformly continuous function as a composition of uniformly continuous functions. Hence 
${^{*}g}_{\nu}(x)$ and $g^{*}_{\nu}(x)$ are well-defined.

\begin{theorem}[McShane-Whitney theorem for $\nu$-continuous functions]\label{thm: lMW} If $(X, A)$ is
a $\nu${-}MW-pair and $g \in S(A, \nu)$, the following hold.\\[1mm]
$(i)$ $g^{*}_{\nu}, {^{*}}g_{\nu} \in S(X, \lambda)$.\\[1mm]
$(ii)$ ${g^{*}_{\nu}}_{|A} = ({^{*}}g_{\nu})_{|A} = g$.\\[1mm]
$(iii)$ $\forall_{f \in S(X, \nu)}(f_{|A} = g \Rightarrow g^{*}_{\nu} \leq f \leq {^{*}}g_{\nu})$.\\[1mm]
$(iv)$ The pair $(g^{*}_{\nu}$, ${^{*}}g_{\nu})$ is the unique pair of functions satisfying conditions $(i)$-$(iii)$. 
\end{theorem}

\begin{proof}
Based on the properties of Definition~\ref{def: lambda}, all steps of the proof of Theorem~\ref{thm: MW}
are repeated for the functions ${^{*}g_{\nu}}, g^{*}_{\nu}$. To show the implication 
$d(x_{1}, a) \leq d(x_{2}, a) + d(x_{1}, x_{2}) \Rightarrow 
\nu(d(x_{1}, a)) \leq \nu(d(x_{2}, a) + d(x_{1}, x_{2})) \leq \nu(d(x_{2}, a)) + \nu(d(x_{1}, x_{2}))$
we use the fact that $\forall_{s, t \in [0, + \infty)}\big(s \leq t \Rightarrow \nu(s) \leq \nu(t)\big).$
\end{proof}

\section{Normed spaces}\label{S:norm}

\subsection{McShane-Whitney pairs in normed spaces}\label{s:normpairs}

Throughout this section $X$ is a normed space with norm $||.||$.

\begin{definition}\label{def: convex} 
 A subset $C$ of $X$ is called \textit{convex}, if 
 $\forall_{x, y \in C}\forall_{\lambda \in (0, 1)}(\lambda x + (1 - \lambda) y \in C)$.
 If $C \subseteq X$ is convex, a function $g: C \rightarrow \Real$ is called \textit{convex}, if
 $\forall_{x, y \in C}\forall_{\lambda \in [0, 1]}(g(\lambda x + (1 - \lambda)y) \leq 
 \lambda g(x) + (1 - \lambda)g(y)),$
 and $g$ is called \textit{concave}, if 
 $\forall_{x, y \in C}\forall_{\lambda \in [0, 1]}(g(\lambda x + (1 - \lambda)y) \geq 
 \lambda g(x) + (1 - \lambda)g(y))$. A function $f: X \rightarrow \Real$ is called \textit{sublinear} if
 it is subadditive and positively homogeneous i.e., if
 $f(x + y) \leq f(x) + f(y)$, and $f(\lambda x) = \lambda f(x)$, for every $x, y \in X$ 
 and $\lambda > 0$, respectively. Similarly, $f$ is called \textit{superlinear}, if it is superadditive 
 i.e., if $f(x + y) \geq f(x) + f(y)$, for every $x, y \in X$, and positively homogeneous.
\end{definition}

\begin{proposition}\label{prp: convex} If $C \subseteq X$ is
 convex, $(X, C)$ is an MW-pair, and $g \in \Lip(C, \sigma)$, the following hold.\\[1mm]
 $(i)$ If $g$ is convex, then ${^{*}}g$ is convex.\\[1mm]
$(ii)$ If $g$ is concave, then $g^{*}$ is concave.
  
\end{proposition}

\begin{proof}
We show only (i), and for (ii) we work similarly. Let $x, y \in X$, and $\lambda \in [0, 1]$. If 
$$C_{x} = \{\lambda(g(c) + \sigma ||x-c||) \mid c \in C\},$$ 
$$\ \ \ \ \ \ \ D_{y} = \{(1 - \lambda)(g(c) + \sigma||y - c||) \mid c \in C\},$$
an element of $C_{x} + D_{y}$ has the form $\lambda(g(c) + \sigma ||x-c||) + (1 - \lambda)(g(d) + 
\sigma||y - d||)$, for some
$c, d \in C$. We show that 
$${^{*}}g(\lambda x + (1 - \lambda) y) \leq \lambda(g(c) + \sigma ||x-c||) + (1 - \lambda)(g(d) + 
\sigma||y - d||),$$
 where $c, d \in C$.
Since $C$ is convex, $c{'} := \lambda c + (1 - \lambda) d \in C$, and by the convexity of $g$ we get
\begin{align*}
 {^{*}}g(\lambda x + (1 - \lambda) y) & \leq g(c{'}) + \sigma ||\lambda x + (1 - \lambda) y - c{'}|| \\
 & \leq \lambda g(c) + (1-\lambda)g(d) + \lambda \sigma ||x - c|| + (1 - \lambda) \sigma ||y - d||\\
 & = \lambda(g(c) + \sigma ||x-c||) + (1 - \lambda)(g(d) + \sigma||y - d||).
\end{align*}
Since the element of $C_{x} + D_{y}$ considered is arbitrary, we get that
${^{*}}g(\lambda x + (1 - \lambda) y) \leq \inf (C_{x} + D_{y}) = \inf C_{x} + \inf D_{y} =
\lambda {^{*}}g(x) + (1 - \lambda) {^{*}}g(y).$
\end{proof}


The next proposition says that if $g$ is linear, then ${^{*}}g$ is sublinear and $g^{*}$ is superlinear.

  \begin{proposition}\label{prp: sublinear}
If $A$ is a non-trivial subspace of $X$ such that 
 $(X, A)$ is an MW-pair, and if $g \in \Lip(A, \sigma)$ is linear, the following hold.\\[1mm]
 $(i)$ $g^{*}(x_{1} + x_{2}) \geq g^{*}(x_{1}) + g^{*}(x_{2})$, and
 ${^{*}}g(x_{1} + x_{2}) \leq {^{*}}g(x_{1}) + {^{*}}g(x_{2})$.\\[1mm]
$(ii)$ If $\lambda > 0$, then $g^{*}(\lambda x) = \lambda g^{*}(x)$ and
 ${^{*}}g(\lambda x) = \lambda {^{*}}g(x)$.\\[1mm]
 $(iii)$ If $\lambda < 0$, then $g^{*}(\lambda x) = \lambda {^{*}}g(x)$ and
 ${^{*}}g(\lambda x) = \lambda g^{*}(x)$.
 \end{proposition}

\begin{proof}
In each case we show only one of the two facts.\\
(i) If $a_{1}, a_{2} \in A$, then
\begin{align*}
 {^{*}}g(x_{1} + x_{2}) & = \inf\{g(a) + \sigma||(x_{1} + x_{2}) - a|| \mid a \in A\}\\
 & \leq g(a_{1} + a_{2}) + \sigma||x_{1} + x_{2} - (a_{1} + a_{2})||\\
 & \leq g(a_{1}) + \sigma ||x_{1} - a_{1}|| + g(a_{2}) + \sigma ||x_{2} - a_{2}||,
\end{align*}
therefore
\begin{align*}
 {^{*}}g(x_{1} + x_{2}) & \leq \inf\{g(a_{1}) + \sigma ||x_{1} - a_{1}|| + g(a_{2}) + \sigma ||x_{2} - a_{2}||
 \mid a_{1}, a_{2} \in A\}\\
 & = \inf\{g(a_{1}) + \sigma ||x_{1} - a_{1}|| \mid a_{1} \in A\} + \inf
 \{g(a_{2}) + \sigma ||x_{2} - a_{2}|| \mid a_{2} \in A\} \\
 & = {^{*}}g(x_{1}) + {^{*}}g(x_{2}).
 \end{align*}
(ii) First we show that ${^{*}}g(\lambda x) \leq \lambda {^{*}}g(x)$. If $a \in A$, then
\begin{align*}
 {^{*}}g(\lambda x) & = \inf\{g(a) + \sigma||\lambda x - a|| \mid a \in A\} \\
 & \leq g(\lambda a) + \sigma||\lambda x - \lambda a||\\
 & = \lambda g(a) + |\lambda|\sigma ||x - a|| \\
 & = \lambda(g(a) + \sigma||x - a||),
\end{align*}
therefore
\begin{align*}
  {^{*}}g(\lambda x) & \leq \inf\{\lambda(g(a) + \sigma||x - a||) \mid a \in A\}\\
  & = \lambda \inf\{g(a) + \sigma||x - a|| \mid a \in A\}\\
  & = \lambda {^{*}}g(x).
\end{align*}
For the inequality ${^{*}}g(\lambda x) \geq \lambda {^{*}}g(x)$ we work as follows.
\begin{align*}
 \lambda {^{*}}g(x) & = \lambda \inf\{g(a) + \sigma ||x - a|| \mid a \in A\}\\
 & \leq \lambda(g(\frac{1}{\lambda}a) + \sigma||x - \frac{1}{\lambda}a||)\\
& =  g(a) + \sigma|\lambda|||x - \frac{1}{\lambda}a||\\
 & = g(a) + \sigma||\lambda(x - \frac{1}{\lambda}a)||\\
 & = g(a) + \sigma||\lambda x - a||,
\end{align*}
therefore 
\begin{align*}
\lambda {^{*}}g(x) &  \leq \inf\{g(a) + \sigma||\lambda x - a|| \mid a \in A\} = {^{*}}g(\lambda x).
\end{align*}
(iii)  First we show that ${^{*}}g(\lambda x) \leq \lambda g^{*}(x)$. If $a \in A$, then
\begin{align*}
 {^{*}}g(\lambda x) & = \inf\{g(a) + \sigma||\lambda x - a|| \mid a \in A\} \\
 & \leq g(\lambda a) + \sigma||\lambda x - \lambda a||\\
 & = \lambda g(a) + |\lambda|\sigma ||x - a|| \\
 & = \lambda(g(a) - \sigma||x - a||),
\end{align*}
therefore
\begin{align*}
  {^{*}}g(\lambda x) & \leq \inf\{\lambda(g(a) - \sigma||x - a||) \mid a \in A\}\\
  & = \lambda \sup\{g(a) - \sigma||x - a|| \mid a \in A\}\\
  & = \lambda g^{*}(x).
\end{align*}
For the inequality ${^{*}}g(\lambda x) \geq \lambda g^{*}(x)$ we work as follows. Since 
\begin{align*}
 g^{*}(x) & = \sup\{g(a) - \sigma ||x - a|| \mid a \in A\}\\
 & \geq g(\frac{1}{\lambda}a) - \sigma||x - \frac{1}{\lambda}a||
\end{align*}
and $\lambda < 0$, we get
\begin{align*}
 \lambda g^{*}(x) & \leq \lambda(g(\frac{1}{\lambda}a) - \sigma||x - \frac{1}{\lambda}a||)\\
& =  g(a) - \lambda \sigma ||x - \frac{1}{\lambda}a||\\
& =  g(a) + \sigma |\lambda| ||x - \frac{1}{\lambda}a||\\
 & = g(a) + \sigma||\lambda(x - \frac{1}{\lambda}a)||\\
 & = g(a) + \sigma||\lambda x - a||,
\end{align*}
therefore
\begin{align*}
 \lambda g^{*}(x) &  \leq \inf\{g(a) + \sigma||\lambda x - a|| \mid a \in A\} = {^{*}}g(\lambda x).
                    \tag*{\qedhere}
\end{align*}
\end{proof}

As we have already said in subsection~\ref{s:metric}, if $x_{0} \in X$, it is not generally the case that 
$\Real x_{0} := \{\lambda x_{0} \mid \lambda \in \Real\}$
is a located subset of $X$. If $||x_{0}|| > 0$ though, $\Real x_{0}$ is a $1$-dimensional subspace of $X$ i.e.,
a closed and located linear subset of $X$ of dimension one (see~\cite{BB85}, p.~307). Of course, $\Real x_{0}$
is a convex subset of $X$.

A standard corollary of the classical Hahn-Banach theorem is that if $x_{0} \neq 0$, there is 
a bounded linear functional $u$ on $X$ such that $||u|| = 1$ and $u(x_{0}) = ||x_{0}||$. Its proof is 
based on the extension of the obvious linear map on $\Real x_{0}$ to $X$ 
through the Hahn-Banach theorem. Next follows a first approach to the translation of this corollary
in Lipschitz analysis. Of course, if $x_{0} \in X$, then 
$||.|| \in \LX$ and $L(||.||) = 1$. To find a Lipschitz function on $X$, which is also the McShane-Whitney extension of 
some Lipschitz function defined on the set of a non-trivial McShane-Whitney pair, we work as follows. 
It is immediate to see that if $x_{0} \in X$ such that $||x_{0}|| > 0$, then 
 $\Iii x_{0} := \{\lambda x_{0} \mid \lambda \in [-1, 1]\}$
is a compact subset of $X$. If we define $g: \Iii x_{0} \rightarrow \Real$ by
$g(\lambda x_{0}) = \lambda ||x_{0}||,$
for every $\lambda \in [-1, 1]$, then $g \in \Lip(\Iii x_{0})$ and $L(g) = 1$; if $\lambda, \mu 
\in [-1, 1]$, then 
$|g(\lambda x_{0}) - g(\mu x_{0})| = |\lambda ||x_{0}|| - \mu ||x_{0}|||
= |\lambda - \mu|||x_{0}|| = ||(\lambda - \mu)x_{0}|| = ||\lambda x_{0} - \mu x_{0}||,$
and since 
$$M_{0}(g) = \bigg \{\sigma_{\lambda x_{0}, \mu x_{0}}(g) = 
 \frac{|g(\lambda x_{0}) - g(\mu x_{0})|}{||\lambda x_{0} - \mu x_{0}||}
 = 1 \mid (\lambda, \mu) \in [-1, 1]_{0}\bigg \},$$
we get that $L(g) = \sup M_{0}(g) = 1$. Since $\Iii x_{0}$ is totally bounded,
by Proposition~\ref{prp: examples}(i) we have that $(X, \Iii x_0)$ is an MW-pair, hence 
the extension ${^{*}}g$ of $g$ is in $\LX$. By Proposition~\ref{prp: sameconstant} we 
have that $L({^{*}}g) = L(g) = 1$. 
%
%
%
%
%
%
%
%

 \begin{theorem}\label{thm: corhb2}
  Let $x_{0} \in X$ such that $||x_{0}|| > 0$. If $(X, \Real x_{0})$
  is a MW-pair, there exists a linear function $g \in \Lip(\Real x_0)$ with $L(g) = 1$, such that 
  ${^{*}}g$ is a sublinear
  Lipschitz function on $X$ with ${^{*}}g(x_{0}) = ||x_{0}||$ and $L({^{*}}g) = 1$, and $g^{*}$ is a superlinear
  Lipschitz function on $X$ with $g^{*}(x_{0}) = ||x_{0}||$ and $L(g^{*}) = 1$.
 \end{theorem}
 
\begin{proof}
As in the proof of the previous remark, the function $g: \Real x_{0} \rightarrow \Real$, defined by
$g(\lambda x_{0}) = \lambda ||x_{0}||,$ for every $\lambda \in \Real$, is in $\Lip(\Real x_{0})$
and $L(g) = 1$. Since $(X, \Real x_{0})$ is an MW-pair, the 
extension ${^{*}}g$ of $g$ is a Lipschitz function, and by Proposition~\ref{prp: sameconstant} we get
$L({^{*}}g) = L(g) = 1$. Since $g$ is linear, by Proposition~\ref{prp: sublinear} we have that
${^{*}}g$ is sublinear. Similarly, the 
extension $g^{*}$ of $g$ is a Lipschitz function, and by Proposition~\ref{prp: sameconstant} we get
$L(g^{*}) = L(g) = 1$. Since $g$ is linear, by Proposition~\ref{prp: sublinear} we have that
$g^{*}$ is superlinear.
\end{proof}


\subsection{Approximate McShane-Whitney theorem for normed spaces}\label{s:aprrox}

Next we show the approximate version of the McShane-Whitney theorem for normed spaces. The technique 
of our proof\footnote{The author would like to thank Hajime Ishihara for bringing this technique to our attention.} 
is found in the proof of the lemma that is necessary to prove the approximate version of the 
constructive Hahn-Banach theorem (Lemma 5.9 in~\cite{BR87}, p.~39). What we get in the conclusion of the 
approximate McShane-Whitney theorem is not a $\sigma$-Lipschitz extension of a given $\sigma$-Lipschitz function $g$, 
but a $(1 + \epsilon)\sigma$-Lipschitz extension of $g$, for every $\epsilon > 0$, where the Lipschitz function $g$ 
is defined on a finite-dimensional subspace of a normed space. Of course, if we consider the trivial subspace $\{0\}$
of $X$, the McShane-Whitney theorem holds trivially.

\begin{theorem}[Approximate McShane-Whitney theorem for normed spaces]\label{thm: aprox}
Let $A$ be a non-trivial, finite dimensional, linear subspace of $X$ and $g \in \Lip(A, 
\sigma)$, for some $\sigma > 0$. If $\epsilon > 0$, there exist functions $g^{*}_{\epsilon}, {^{*}}g_{\epsilon}: X 
\to \Real$ such that:\\[1mm]
$(i)$ $g^{*}_{\epsilon}, {^{*}}g_{\epsilon} \in \Lip(X, (1 + \epsilon)\sigma)$.\\[1mm]
$(ii)$ $({g^{*}_{\epsilon}})_{|A} = ({^{*}}g_{\epsilon})_{|A} = g$.\\[1mm]
$(iii)$ $\forall_{f \in \Lip(X, (1 + \epsilon)\sigma)}(f_{|A} = g \Rightarrow 
g^{*}_{\epsilon} \leq f \leq {^{*}}g_{\epsilon})$.\\[1mm]
$(iv)$ The pair $(g^{*}_{\epsilon}$, ${^{*}}g_{\epsilon})$ is the unique pair of functions satisfying 
conditions $(i)${-}$(iii)$. 

\end{theorem}

\begin{proof}Suppose first that $g(0) = 0$.\\[1mm]
Let $a, b \in A$ and $x \in X$. Since $g(a) - g(b) \leq \sigma ||a -b|| \leq \sigma ||a - x|| + \sigma ||x - b||$,
we have that
$$g(a) - \sigma ||a - x|| \leq g(b) + \sigma ||x - b||,$$
hence
$$\phi_{\epsilon, x}(a) := g(a) - (1 + \epsilon)\sigma ||a - x|| \leq g(b) + (1 + \epsilon)\sigma ||x - b|| =:
\theta_{\epsilon, x}(b).$$
We find $r > 0$, such that for every $a \in A$
\begin{equation}\label{eq: eq1}
||a|| \geq r \Rightarrow \phi_{\epsilon, x}(a) \leq -(1 + \epsilon)\sigma ||x|| = \phi_{\epsilon, x}(0).
\end{equation}
Let
$$r := \frac{2(1 + \epsilon)}{\epsilon}||x||.$$
We show first that 
\begin{equation}\label{eq: eq2}
||a|| \geq r \Rightarrow ||a|| - (1 + \epsilon)||a - x|| \leq -(1 + \epsilon)||x||. 
\end{equation}
We have that 
\begin{align*}
||a - x|| \geq ||a|| - ||x|| & \To -||a - x|| \leq ||x|| - ||a|| \\
& \stackrel{\times (1 + \epsilon)} \To -(1 + \epsilon)||a - x|| \leq (1 + \epsilon)||x|| - (1 + \epsilon)||a|| \\
& \stackrel{+ ||a||} \To ||a||-(1 + \epsilon)||a - x|| \leq (1 + \epsilon)||x|| - (1 + \epsilon)||a|| + ||a|| \\
& \TOT ||a||-(1 + \epsilon)||a - x|| \leq (1 + \epsilon)||x|| - \epsilon ||a||.
\end{align*}
By the definition of $r$ we have that
\begin{align*}
||a|| \geq \frac{2(1 + \epsilon)}{\epsilon}||x|| & \To -||a|| \leq -\frac{2(1 + \epsilon)}{\epsilon}||x||\\
& \stackrel{\times \epsilon} \To -\epsilon ||a|| \leq -2(1 + \epsilon)||x||,
\end{align*}
hence 
$$||a||-(1 + \epsilon)||a - x|| \leq (1 + \epsilon)||x|| - \epsilon ||a|| \leq (1 + \epsilon)||x|| -2(1 + \epsilon)||x||
= -(1 + \epsilon)||x||.$$
Hence, if $||a|| \geq r$, multiplying the conclusion of~(\ref{eq: eq2}) by $\sigma > 0$ we get
$$\sigma||a|| - (1 + \epsilon)\sigma||a - x|| \leq -(1 + \epsilon)\sigma||x||.$$
Since 
$$g(a) = g(a) - g(0) \leq |g(a) - g(0)| \leq \sigma ||a - 0|| = \sigma ||a||,$$
we get 
$$g(a) - (1 + \epsilon)\sigma||a - x|| \leq -(1 + \epsilon)\sigma||x||,$$
which is the required conclusion of~(\ref{eq: eq1}).
Since $A$ is finite dimensional space, the set $K_r := \{a \in A \mid ||a|| \leq r\}$ is compact
(see~\cite{BR87}, p.~34), and
$$s_{\epsilon, x} := \sup \{\phi_{\epsilon, x}(a) \mid a \in K_r\}$$
is well defined, since $\phi$ is uniformly continuous and $K_r$ is totally bounded. Since $0 \in K_r$ 
and since the set $\{\phi_{\epsilon, x}(a) \mid a \in A \ \& \ ||a|| \geq r\}$ is bounded above by $\phi_{\epsilon, x}(0)$, 
we get\footnote{Here 
we use the simple fact that if $D \subseteq B \subseteq \Real$ such that $D$ is dense in $B$ and $\sup D$ 
exists, then $\sup B$ exists and
$\sup B = \sup D$. The set $E = K_r \cup \{a \in A \mid ||a|| \geq r\}$ is dense in $A$ and $\phi_{\epsilon, x}(E)$ 
is dense in $\phi_{\epsilon, x}(A)$.} 
$$s_{\epsilon, x} = \sup \{\phi_{\epsilon, x}(a) \mid a \in A\}.$$ 
Hence we define
$$g^{*}_{\epsilon}(x) := s_{\epsilon, x} = \sup \{g(a) - (1 + \epsilon)\sigma||x - a|| \mid a \in A\}.$$
Next we show that $g^{*}_{\epsilon} \in \Lip(X, (1+ \epsilon)\sigma)$. Since $||x-a|| \leq ||x-y|| + ||y-a||$, we get
$g(a) - (1 + \epsilon)\sigma||x - a|| \geq g(a) - (1 + \epsilon)\sigma||y - a|| - (1 + \epsilon)\sigma||x - y||$, and 
$(1 + \epsilon)\sigma||x - y|| + g(a) - (1 + \epsilon)\sigma||x - a|| \geq g(a) - (1 + \epsilon)\sigma||y - a||$, hence
$(1 + \epsilon)\sigma||x - y|| + g(a) - (1 + \epsilon)\sigma||x - a|| \geq g^{*}_{\epsilon}(y)$, and
$(1 + \epsilon)\sigma||x - y|| + g^{*}_{\epsilon}(x) \geq g^{*}_{\epsilon}(y)$. Consequently, 
$$g^{*}_{\epsilon}(y) - g^{*}_{\epsilon}(x) \leq (1 + \epsilon)\sigma||x - y||.$$
Working similarly, we also get $g^{*}_{\epsilon}(x) - g^{*}_{\epsilon}(y) \leq (1 + \epsilon)\sigma||x - y||.$
Suppose next that $a_0 \in A$. By definition we have that $g^{*}_{\epsilon}(a_0) = \sup \{g(a) - (1 + \epsilon)||a_0 - a|| 
\mid a \in A\}.$ Clearly, $g(a_0) = g(a_0) - (1 + \epsilon)\sigma||a_0 - a_0|| \leq g^{*}_{\epsilon}(a_0)$.
If $a \in A$, then $g(a) - g(a_0) \leq \sigma ||a - a_0|| \leq (1 + \epsilon)\sigma||a - a_0||$, hence
$- g(a_0) \leq - g(a) + (1 + \epsilon)\sigma||a - a_0||$ and $g(a_0) \geq g(a) - (1 + \epsilon)\sigma||a - a_0||$. Since
$a \in A$ is arbitrary, we get $g(a_0) \geq g^{*}_{\epsilon}(a_0)$. If $f \in \Lip(X, (1 + \epsilon)\sigma)$ such that
$f_{|A} = g$, then $g(a) -f(x) = f(a) - f(x) \leq (1 + \epsilon)\sigma||x - a||$, hence 
$g(a) - (1 + \epsilon)\sigma||x - a|| \leq f(x)$, and since $a \in A$ is arbitrary, we get $g^{*}_{\epsilon}(x) \leq f(x)$.\\[1mm]
For the construction of ${^{*}}g_{\epsilon}$ we work similarly, as follows. The same $r > 0$ that we chose earlier satisfies,
for every $a \in A$, also the property
$$||a|| \geq r \Rightarrow \theta_{\epsilon, x}(a) = g(a) + (1 + \epsilon)\sigma||x - a|| \geq (1+ \epsilon)\sigma||x||
= \theta_{\epsilon, x}(0).$$
The reason for this is that since we can choose $||a|| \geq r$ such that $||a|| \leq (1 + \epsilon)(||a - x|| - ||x||)$,
we have that 
\begin{align*}
& -||a|| \geq (1 + \epsilon)(||x|| -||a - x||) \Rightarrow \\
& -||a|| + (1 + \epsilon)(||a - x||) \geq (1 + \epsilon)||x|| \Rightarrow \\
& - \sigma||a|| + (1 + \epsilon)\sigma(||a - x||) \geq (1 + \epsilon)\sigma||x|| \Rightarrow \\
& g(a)  + (1 + \epsilon)\sigma(||a - x||) \geq (1 + \epsilon)\sigma||x||,
\end{align*}
since from the inequality $-g(a) \leq |g(a)| \leq \sigma ||a||$, we get $g(a) \geq - \sigma ||a||$. 
Similarly,
$$\iota_{\epsilon, x} := \inf \{\theta_{\epsilon, x}(a) \mid a \in K_r\}$$
is well-defined. Since the set 
$\{\theta_{\epsilon, x}(a) \mid a \in A \ \& \ ||a|| \geq r\}$ is bounded below by $\theta_{\epsilon, x}(0)$, we get
$$\iota_{\epsilon, x} = \inf \{\theta_{\epsilon, x}(a) \mid a \in A\}.$$
Hence, we define
$${^{*}}g_{\epsilon}(x) := \iota_{\epsilon, x} = \inf \{g(a) + (1 + \epsilon)\sigma ||x - a|| \mid a \in A\}.$$
To show the properties of ${^{*}}g_{\epsilon}$ in (i)-(iii), we proceed as we did in the proof of the properties 
of $g^{*}_{\epsilon}$.  
The uniqueness of the pair $(g^{*}_{\epsilon}, {^{*}}g_{\epsilon})$ follows immediately.\\[1mm]
In the general case, where we do not assume that $g(0) = 0$, we work as follows. If $h : A \to \Real$ is defined by
$$h(a) := g(a) - g(0),$$
for every $a \in A$, then $h(0) = 0$. Clearly, $h \in \Lip(A, \sigma)$, hence by our proof above there 
are functions $h^{*}_{\epsilon}, {^{*}}h_{\epsilon}: X \to \Real$ such that:\\[1mm]
(i)$^{'}$ $h^{*}_{\epsilon}, {^{*}}h_{\epsilon} \in \Lip(X, (1 + \epsilon)\sigma)$.\\[1mm]
(ii)$^{'}$ $({h^{*}_{\epsilon}})_{|A} = ({^{*}}h_{\epsilon})_{|A} = h$.\\[1mm]
(iii)$^{'}$ $\forall_{F \in \Lip(X, (1 + \epsilon)\sigma)}(F_{|A} = h \Rightarrow 
h^{*}_{\epsilon} \leq F \leq {^{*}}h_{\epsilon})$.\\[1mm]
(iv)$^{'}$ The pair $(h^{*}_{\epsilon}$, ${^{*}}h_{\epsilon})$ is the unique pair of functions satisfying 
conditions (i)$^{'}{-}$(iii)$^{'}$. \\[1mm]
We define $g^{*}_{\epsilon}, {^{*}}g_{\epsilon}: X \to \Real$ by
$$g^{*}_{\epsilon}(x) := h^{*}_{\epsilon}(x) + g(0),$$
$${^{*}}g_{\epsilon}(x) := {^{*}}h_{\epsilon}(x) + g(0),$$
for every $x \in X$. Properties (i)-(iv) for $g^{*}_{\epsilon}$ and ${^{*}}g_{\epsilon}$ follow immediately
from properties (i)$^{'}{-}$(iv)$^{'}$ for $h^{*}_{\epsilon}$ and ${^{*}}h_{\epsilon}$.
\end{proof}

\section{Concluding remarks and open questions}\label{S:open}

In this paper we examined the McShane-Whitney extension of Lipschitz functions from the Bishop-style,
constructive point of view. As in the case of Tietze's theorem, or the Hahn-Banach theorem,
the generality of the classical McShane-Whitney theorem is lost. In reward though, 
the computational content of this extension is constructively studied through the notion of a
McShane-Whitney pair. 

Moreover, the analogies between the Hahn-Banach theorem and the McShane-Whitney theorem made possible to 
transfer methods and results from one to the other. A significant example of the interplay between the two
theorems is the approximate McShane-Whitney theorem, the proof of which uses the proof-technique of the 
approximate Hahn-Banach theorem.

Another example that we plan to elaborate in future work is the following.
If $p$ is a \textit{seminorm} on a vector space $X$ i.e.,  
$p : X \to \Real$ such that (i) $p(x) \geq 0$, (ii) $p(x + y) \leq p(x) + p(y),$ and (iii) $p(\lambda x) = 
|\lambda| p(x)$, for
every $x, y \in X$ and $\lambda \in \Real$, the set of $p$-Lipschitz functions on $X$ is defined by
$$p{-}\Lip(X) := \bigcup _{\sigma \geq 0} p{-}\Lip(X, \sigma),$$
$$p{-}\Lip(X, \sigma) := \{f \in \mathbb{F}(X, Y) \mid \forall_{x, y \in X}\big(|f(x) -  f(y)| \leq 
\sigma p(x - y)\big)\}.$$ 
One can define a $p$-MW-pair in the obvious way. In order to get examples of $p${-}MW-pairs, one can define 
$p$-versions of locatedness, total boundedness and uniform continuity, and prove, similarly, the $p$-versions
of the propositions in Bishop and Bridges pp.~94-95 and the $p$-version of Proposition~\ref{prp: examples}.
The proof of the expected McShane-Whitney theorem for $p$-Lipschitz functions is similar to the proof of 
Theorem~\ref{thm: MW}.
The following MW-version of the analytic Hahn-Banach theorem for seminorms\footnote{According to it, if $p$ 
is a seminorm on $X$, $A$ a subspace of $X$ and $g : A \to \Real$ a linear function such that 
$\forall_{a \in A}\big(g(a) \leq p(a)\big)$, there is linear function $f : X \to \Real$ that extends $g$ and 
$\forall_{x \in X}\big(f(x) \leq p(x)\big)$. In its general formulation, the analytic Hahn-Banach theorem refers
to sublinear functionals $p$, where the condition (iii) of a seminorm is replaced by the condition 
$p(\lambda x) = \lambda p(x)$, for $\lambda \geq 0$ (see~\cite{Li16}, p.~120). Its proof uses Zorn's lemma.}
can be shown.\\[2mm]
\textit{If $p$ is a seminorm on
$X$, $(X, A)$ is a $p$-MW-pair, $A$ is a subspace of $X$, and $g : A \to \Real$ is linear, then if
$\forall_{a \in A}(g(a) \leq p(a))$, there is a superlinear extension $g^{*}$ of $g$ such that $g^{*} 
\in \pLip(X, 1)$ and $\forall_{x \in X}(g^{*}(x) \leq p(x))$, while if 
$\forall_{a \in A}(g(a) \geq p(a))$, there is a sublinear extension ${^{*}}g$ of $g$ such that 
 ${^{*}}g \in \pLip(X, 1)$ and $\forall_{x \in X}({^{*}}g(x) \geq p(x))$}.\\[2mm]
With respect to the analytic Hahn-Banach theorem for seminorms we lose linearity of the extension function, since 
in the case (i) of the previous result $g^{*}$ is only superlinear, but instead we have that $g^{*} \in \pLip(X, 1)$, and we avoid
Zorn's lemma. We expect to find good applications of the previous result in convex analysis, where linearity
is not as important as sublinearity, or superlinearity.\\

Moreover, the following problems and questions arise naturally from our work.

\begin{enumerate}


\item To find necessary and sufficient conditions on metric spaces $X, Y$ and on function $f \in \Lip(X, Y)$ 
such that $L(f)$ and/or $L^*(f)$ exist. This problem is similar to the normability of a bounded linear functional $T$, 
which, constructively is equivalent to the locatedness of the kernel of $T$. 
\item To characterise, for given metric space $X$, its MW-pairs $(X, A)$.
\item To find necessary and/or sufficient conditions on a normed space $(X, ||.||)$ 
such that $(X, \Real x_{0})$ is an MW-pair for some given $x_0 \in X$ with $||x_{0}|| > 0$. We expect that 
the properties added by Ishihara in~\cite{Is89} to the hypotheses of the Hahn-Banach theorem i.e., the uniform
convexity and separability of $X$ and the G\^{a}teaux differentiability of its norm, are interesting properties 
to be tested for this question too. 
\item To connect the separation version of the Hahn-Banach theorem to the McShane-Whitney theorem.
\end{enumerate}


\vspace{3mm}

We hope to address these questions and problems in future work.

\vspace{8mm}

\noindent
The first version of this paper was completed during our research visit to the Japan Advanced Institute of 
Science and Technology (JAIST) that was funded by the European-Union project ``Computing with Infinite Data''.
We would like to thank Hajime Ishihara for hosting us in JAIST and discussing with us the material of this work. 
The second version of this paper was completed during our research visit to the Philosophy Department of
Carnegie Mellon University that was funded by the same European-Union project. We would like to thank Wilfried Sieg
for hosting us in CMU.
\vspace{2mm}

\noindent
We would also like to thank the anonymous referees for their very helpful comments and suggestions.

\end{document}